\documentclass[10pt, reqno]{amsart}
\usepackage{graphicx} 
\usepackage{amsthm}
\usepackage{amsmath}
\usepackage{amssymb}
\usepackage{amsfonts}
\usepackage{fullpage}
\usepackage{bbm}
\usepackage{xcolor}
\usepackage{mathtools}
\usepackage{hyperref}

\newtheorem{theorem}{Theorem}[section]
\newtheorem{proposition}{Proposition} [section]
\newtheorem{definition}{Definition}[section]
\newtheorem{lemma}{Lemma}[section]

\numberwithin{equation}{section}
\newtheorem{remark}{Remark}[section]

\title{Global Approximate Controllability of the Camassa-Holm Equation by a finite dimensional Force}
\author{Shirshendu Chowdhury,  Rajib Dutta and  Debanjit Mondal}

\begin{document}
\begin{abstract}
In this paper, we consider the Camassa-Holm equation posed on the periodic domain $\mathbb{T}.$ We show that Camassa-Holm equation is globally approximately controllable by three dimensional external force in $H^s(\mathbb{T})$ for $s > \frac{3}{2}.$ The proof is based on Agrachev-Sarychev approach in geometric control theory.
\end{abstract}

\keywords{Camassa-Holm equation. Approximate controllability. Geometric control theory}
\subjclass{ 93B05; 93C20; 76B15;  }
\allowdisplaybreaks

\maketitle
\section{Introduction}
In this paper, we are interested in control problem concerning the Camassa-Holm equation on the circle $\mathbb{T} := \mathbb{R}/ 2\pi \mathbb{Z}$.
\begin{align}\label{main_equation}
u_{t} - u_{txx} + 2\kappa u_{x} + 3uu_{x} = 2u_{x}u_{xx} + uu_{xxx}  \hspace{1cm} \text{ for } \ (t , x) \in (0, T) \times \mathbb{T}.
\end{align}
The Camassa-Holm equation describes one-dimensional surface waves at a free surface of shallow water under the influence of gravity. The function $u(t,x)$ represents the fluid velocity at time $t$ and position $x$, and the constant $\kappa$ is a non negative parameter in \eqref{main_equation} . The equation was first introduced by Fokas and Fuchssteiner \cite{Fokas_81} as a bi-hamiltonian model, and was derived as a water wave model by Camassa and Holm \cite{Camassa_Holm_93}. Moreover one can describe the periodic Camassa-Holm equation as the geodesic equation on the diffeomorphism group of the circle or on the Bott-Virasoro group (see Misiolek \cite{Misiolek_2002}).

\vspace{5mm}
We consider the following control problem $\kappa = \frac{1}{2}$ for simplicity in calculation; the same can be done for general $\kappa.$
\begin{equation}\label{ctrl_equn}
\begin{cases}
u_{t} - u_{txx} + u_x + 3uu_{x} = 2u_{x}u_{xx} + uu_{xxx} + \eta(t,x)  \hspace{1cm} \text{for}\ (t,x) \in (0,T) \times \mathbb{T},\\ u(0,x) = u_{0}(x),    
\end{cases}
\end{equation}

where $T > 0$, $u_0$ is the initial value, and $\eta$ is a control. We will discuss the approximate controllability of \eqref{ctrl_equn}. More precisely, it will be  proved that for any given $u_0,u_1$ in some suitable spaces, we can find a finite-dimensional control $\eta$ such that the solution of \eqref{ctrl_equn}  can be steered to an arbitrary small neighborhood of $u_1$ in time $T$ starting from $u_0$. There is no restriction on control time $T$ and the amplitude of $u_0,u_1$. Although there are lots of existing results of controllability of Camassa-Holm Equation, such as in this paper \cite{Oliver_08} the author proved the exact controllability of \eqref{main_equation} with localized interior control but in our case the control $\eta$ takes values in a finite-dimensional space. This kind of control has theoretical significance and wide application in physics and engineering. To obtain the desired result, we adopt the Agrachev-Sarychev approach.

Let us define $\Lambda^s := (1 - \partial_{xx})^{\frac{s}{2}}$, the pseudo-differential operator $\Lambda^s $ is defined for any $s \in \mathbb{R}$ on a test function $f$ by $$\widehat{\Lambda^s f}(\xi) = (1 + \xi^2)^{\frac{s}{2}} \hat{f}(\xi), $$ where $\hat{f}$ denotes the Fourier transformation of a function $f$ on the circle $\mathbb{T} = \mathbb{R}/2\pi \mathbb{Z}$, for $\xi \in \mathbb{Z}$ $$\hat{f}(\xi) = \int_{\mathbb{T}} e^{-i\xi x} f(x) dx . $$ 
Also, we recall that for any $s \in \mathbb{R}$ the sobolev space $H^s = H^s(\mathbb{T})$ is defined by $$H^s(\mathbb{T}) = \left\{ f \in \mathcal{D}'(\mathbb{T}) : \|f\|_{H^s} = \|\Lambda^s f\|_{L^2} \simeq \left(\sum_{\xi \in \mathbb{Z}} \left(1 + \xi^2\right)^{s} |\hat{f}(\xi)|^2\right)^{\frac{1}{2}} < \infty \right\}.$$

It is known fact that the periodic Camassa-Holm Equation is well-posed in $H^s$ for $s > \frac{3}{2},$ Danhin proved the well-posedness in the paper \cite{Danchin_2001} using Besov space and in the paper \cite{Himonas_2001} Himonas and Misiolek proved the same using Friedrichs mollifier. See also \cite{de_Lellis_07}.

\begin{remark}\label{S>3/2}
If $ s > \frac{3}{2} $ and $ f \in H^s(\mathbb{T}),  \  $ then $ f_x \in L^{\infty}(\mathbb{T}),$ this fact is crucially used to prove wellposedness of our the control system in Section $ 4.$   
\end{remark}
That's why we are looking for Approximate controllability result in $H^s$, $s > \frac{3}{2}.$
 
 \begin{definition}\label{approxi_cntrl}
We say the equation \eqref{ctrl_equn} is approximately controllable in $H^s(\mathbb{T})$ by values in $\mathcal{H}_0$ if for any $T > 0, \varepsilon > 0$ and any $u_{0}, u_1 \in H^s(\mathbb{T})$, there is a piece wise constant control with values in $\mathcal{H}_0$ and a solution $u$ of \eqref{ctrl_equn} such that $$\| u(T) - u_1\|_{H^s} \le \varepsilon.$$
 \end{definition}
Our main result is the following theorem. 
\begin{theorem}\label{main_thm}
For $ s > \frac{3}{2},$ equation \eqref{ctrl_equn} is approximately controllable in $H^s(\mathbb{T})$ by a piece wise constant controls with values in $\mathcal{H} ,$ where $$\mathcal{H} = \text{span}\{1, \sin(x), \cos(x)  \}.$$
\end{theorem}

Approximate controllability of PDEs by additive finite-dimensional forces has been studied by many authors in the recent years. The first results are obtained by Agrachev and Sarychev \cite{Agrachev_Sarychev_05, Agrachev_Sarychev_06}, who considered the Navier-Stokes and Euler systems on the two dimensional tours. Their approach has been generalized by Shirikyan \cite{Shirikyan_06, Shirikyan_07} to the case of three dimensional Navier-Stokes system; see also the papers \cite{Shirikyan_14, Shirikyan_17} by Shirikyan, where the Burgers equation is considered on real line and  the bounded interval with Dirichlet boundary conditions.  In the periodic setting, Nersisyan \cite{Nersisyan_10,Nersisyan_11} considered three dimensional Euler systems for perfect compressible and incompressible fluids, Sarychev \cite{Sarychev_12} studied the two dimensional cubic Schrodinger equation.

The proof of the Theorem $1.1$ is based on a technique of applying large controls on short time intervals. Previously, such ideas have been used mainly in the studied of finite-dimensional control system; example, see the works of Jurdjevic and Kupka \cite{Jurdjevic_85, Jurdjvic_97}.  Then infinite-dimensional extensions of this technique appear in the above-cited papers of Agrachev-Sarychev. More recently this approach has been used in the paper Glatt-Holtz, Herzog, and Mattingly \cite{Glatt_18}, where , in particular, a 1D parabolic PDE is considered with polynomial nonlinearity of odd degree, and Shirikyan did for Burger equation with Dirichlet boundary condition See \cite{Shirikyan_18} then in the paper of Narsesyan \cite{Narsesyan_21}, where the nonlinearity is a smooth function that grows polynomially without any restriction on the degree and on the space dimension. Then the authors Mo Chen in the paper \cite{Chen23} and the author Melek Jellouli in the paper \cite{Jellouli_23} using same technique prove the result for Korteweg-de Vries Equation and BBM Equation respectively.

As we have discussed, the idea of the proof of the main theorem is motivated by many recent works related
Agrachev-Sarychev method. However, to use the ideas in the Camassa-Holm equation, we will encounter some
difficulties that demand special attention, and some new tools will be needed. we first prove that trajectory of \eqref{ctrl_equn} can be steered close to any target $u_1$ belongs to the set $u_0 + \mathcal{H}_0$ in small time, where $\{\mathcal{H}_j \}_{j \ge 0}$ is a non-decreasing sequence of subspaces defined in Section $2.$  By an iterating argument, we show that starting $u_0,$ the trajectory can attain approximately any point in $u_0 + \mathcal{H}_N$ for any $N \in \mathbb{N}.$ In this step, the key point is the following asymptotic property  
$$u(\cdot , \delta) \rightarrow u_0 - \varphi\varphi_x + (1 - \partial_{xx})^{-1}\left(\eta - 2\varphi\varphi_x - \varphi_x\varphi_{xx}\right), \text{in}\ H^s(\mathbb{T}) \text{ as }\ \delta \rightarrow 0^+. $$
where $ u $ is the solution of 
\begin{align*}
\begin{cases}
u_{t}-u_{txx} + (u+\delta^{-\frac{1}{2}}\varphi)_x + 3(u+\delta^{-\frac{1}{2}}\varphi)(u+\delta^{-\frac{1}{2}}\varphi)_{x} - 2(u+\delta^{-\frac{1}{2}}\varphi)_{x}(u+\delta^{-\frac{1}{2}}\varphi)_{xx}\\
\hspace{6cm}- (u+\delta^{-\frac{1}{2}}\varphi)(u+\delta^{-\frac{1}{2}}\varphi)_{xxx} = \delta^{-1} \eta \\
u(0,x)=u_0(x).
\end{cases}
\end{align*}
Then by the fact that $\bigcup_{n = 1}^{\infty} \mathcal{H}_{n - 1}$ is dense in $H^s(\mathbb{T}),$ we can see that system \eqref{ctrl_equn} is approximately controllable in small time. Finally, applying the well-posedness and stability of \eqref{ctrl_equn}, we can keep the trajectory close to terminal state $u_1$ for any time $T.$ Remaining of the paper organised as follows:
\begin{itemize}
    \item In Section $2,$ we state the required propositions and prove density of $\mathcal{H}_N$ in $H^s$ with algebraic property of resolvent map. 
    \item In Section $3$ we prove the Theorem \ref{main_thm}.
    \item Section $4$ is devoted to the proposition used in the proof of Theorem \ref{main_thm}.
    \item Finally in section $ 5 $ we have constructed a explicit control for a simple case.
\end{itemize}

\vspace{5mm}
\section{Existence and properties.} For the technical reasons that we will see below, we introduce a smooth function $\varphi(x) $. So we consider the following Camassa-Holm Equation on torus
\begin{align} \label{4}
\begin{cases}
u_{t}-u_{txx} + (u+\varphi)_x + 3(u+\varphi)(u+\varphi)_{x} = 2(u+\varphi)_{x}(u+\varphi)_{xx} + (u+\varphi)(u+\varphi)_{xxx}+\Tilde{\eta}(t,x)  \\ \hspace{10cm}\text{ for }\  t > 0 , x \in  \mathbb{T}\\
u(0, x) = u_0(x)    
\end{cases}
\end{align}

where $\Tilde{\eta}(t) \in L^2(\mathbb{T}) $. For $u_0 \in H^s(\mathbb{T})$, the solution of \eqref{4} at time $t$, with a control term $\Tilde{\eta}$, is denoted $u(t) = \mathcal{R}_{t}(u_0, \varphi, \Tilde{\eta}) $. Note that the function $v = u + \varphi $ is solution of the equation \eqref{ctrl_equn} with initial condition $ u_0 + \varphi.$ The equation \eqref{4} can be written under the form 
\begin{align}\label{inverse_equn}
\begin{cases}
u_t = A(u+\varphi)-(u+\varphi)\partial_x(u+\varphi)-(1-\partial_{xx})^{-1}\Big[2(u+\varphi)\partial_x (u+\varphi) + \partial_x(u+\varphi)\partial_{xx}(u+\varphi)\Big] + f\\
u(0, x) = u_0(x)
\end{cases}
\end{align}
Where $A = -\Lambda^{-2}\partial_x = - (1 - \partial_{xx})^{-1} \partial_x $ and $f = \Lambda^{-2}\Tilde{\eta} = (1 - \partial_{xx})^{-1}\Tilde{\eta}.$ We know that Since $A$ is bounded then it is the infinitesimal generator of a uniformly continuous semigroup $\{e^{tA}\}_{t\ge 0}.$ 

We start by studying the existence of the solution as well as some estimations that we will use in following. Let $\delta > 0, \varphi(x) \ \text{and}\ f(x)$ two smooth functions, we consider the equation

\begin{equation}\label{gen_equn}
\begin{cases}
u_{t}-u_{txx} + (u+\delta^{-\frac{1}{2}}\varphi)_x + 3(u+\delta^{-\frac{1}{2}}\varphi)(u+\delta^{-\frac{1}{2}}\varphi)_{x} - 2(u+\delta^{-\frac{1}{2}}\varphi)_{x}(u+\delta^{-\frac{1}{2}}\varphi)_{xx}\\
\hspace{4cm}- (u+\delta^{-\frac{1}{2}}\varphi)(u+\delta^{-\frac{1}{2}}\varphi)_{xxx} = \delta^{-1}f\ \hspace{.5cm} \text{ for }\  t > 0, x \in  \mathbb{T},\\
u(0,x)=u_0(x).
\end{cases}
\end{equation}

\begin{proposition} [Well-posedness]\label{prp_existance}
 For $ s > \frac{3}{2},\ u_0  \in H^s(\mathbb{T})$ , $\varphi \in H^{s + 1}(\mathbb{T})$ and $f \in L^2_{loc}(R^+; H^{s - 2}(\mathbb{T}))$ there exists time $0 < T_{*} := T_{*}(u_0, \varphi, f) $ such that system \eqref{inverse_equn} admits a unique solution $u \in C([0,T_{*}]; H^s(\mathbb{T})).$
\end{proposition}

\begin{proposition}[Stability]\label{prp_cntity}
For $s > \frac{3}{2}$ and given  $u_0, v_0,  \in H^{s + 1}(\mathbb{T}),$ and $g \in L^2_{loc}(R^+; H^{s - 2}(\mathbb{T}))$ there exists time $T > 0 $ and  constants $c,$ such that, for all $t \le T$
\begin{align}\label{lips_property}
\|\mathcal{R}_t(u_0, 0, g) - \mathcal{R}_t(v_0, 0, g)\|_{H^s}\le c \|u_0 - v_0\|_{H^s} 
\end{align}
\end{proposition}

\begin{remark}\label{rmrk_uniqness}
From the uniqueness of the solution we have the equality for all $t \in [0, \delta]$
\begin{equation}\label{uniqueness_property}
\mathcal{R}_t(u_0, \varphi, \eta) = \mathcal{R}_t(u_0 + \varphi, 0, \eta) - \varphi
\end{equation}
i.e solution of the equations \eqref{ctrl_equn} and\eqref{4}  are related by this equation.   
\end{remark} 

For any subspace $G,$ we denote the space
$$\mathcal{F}(G) := \text{span} \Bigg\{\eta - \sum _{i=1}^{d} \varphi_{i}\partial_x\varphi_i - (1 - \partial_{xx})^{-1} \sum _{i=1}^{d} \big(2\varphi_i\partial_x\varphi_i + \partial_x \varphi_i \partial_{xx}\varphi_i \big) ; \eta, \varphi_i \in G, \forall d \geq 1\Bigg\}. $$
 Note that this space $\mathcal{F}(G)$ is defined through the nonlinear term present in the equation.

\newpage
Then we can construct a sequence of finite-dimensional spaces:

$$\mathcal{H}_0 = \mathcal{H} \text{ , } \mathcal{H}_n = \mathcal{F}(\mathcal{H}_{n - 1}) \text{ , } n \ge 1 \text{ , } \mathcal{H}_{\infty} = \bigcup_{n = 1}^{\infty} \mathcal{H}_{n - 1}  $$

\begin{definition}\label{saturating}
We say that $\mathcal{H}$ is saturating if $\mathcal{H}_{\infty}$ is dense in $H^s(\mathbb{T}).$
\end{definition}

\begin{proposition}[Density]\label{prp_saturating}
The space $\mathcal{H}$ is saturating.
\end{proposition}
\begin{proof}
It is clear from the definition of $\mathcal{H}_{n - 1}(n \ge 1)$ that $\mathcal{H}_0 \subset \mathcal{H}_1 \subset \text{ ... } \subset \mathcal{H}_n \subset \textit{ ... }.$ Thus the above will be proved if we can show 
\begin{equation}\label{ind_subspace}
\sin(mx), \cos(mx) \in \mathcal{H}_{m - 1}, \forall m \ge 1.   
\end{equation}
We prove by the mathematical induction. Before applying induction observe that $(1 - \partial_{xx}) \sin (mx) = (1 + m^2)\sin (mx) $ and $(1 - \partial_{xx}) \cos (mx) = (1 + m^2)\cos (mx) $ i.e for each $m \in \mathbb{N} $ the spaces $span\{\sin (mx)\}$ and $span\{\cos (mx)\}$ are invariant under $(1 - \partial_{xx})$ and $(1 - \partial_{xx})^{-1}.$ Now for $m = 1$, \eqref{ind_subspace} is obvious. For $ m = 2,$ take $ \eta = 0, \varphi = \sin (x) \in \mathcal{H}_0$ then
\begin{align*}
\eta - \varphi \partial_x \varphi - (1 - \partial_{xx})^{-1}\Big(2\varphi \partial_x \varphi + \partial_x \varphi \partial_{xx} \varphi \Big) = - \frac{3}{5} \sin(2x) \in \mathcal{H}_1
\end{align*}
and taking $\eta = 0, \varphi = \left(\sin(x) + \cos(x)\right) \in \mathcal{H}_0,$ we have
\begin{align*}
\eta - \varphi \partial_x \varphi - (1 - \partial_{xx})^{-1}\Big(2\varphi \partial_x \varphi + \partial_x \varphi \partial_{xx} \varphi \Big) = - \frac{6}{5} \cos(2x) \in \mathcal{H}_1
\end{align*}
Now assuming  $\sin(mx), \cos(mx) \in \mathcal{H}_{m - 1}$ our aim to show  $\sin((m + 1)x), \cos((m + 1)x) \in \mathcal{H}_{m}$
We set 
\begin{align*}
&\varphi_1 =  \cos (x) + \sin (x),\\
&\varphi_2 = - \cos (x) + \sin (x),\\
&\varphi_3 = \alpha \cos (mx) + \beta \sin (mx) + \cos (x) + \sin (x),\\
&\varphi_4 = -\beta \cos (mx) + \alpha \sin (mx) + \cos (x) - \sin (x).
\end{align*}
 with $\eta = 0,$ then $\varphi_1, .., \varphi_4 \in \mathcal{H}_{m - 1}$
Consider
\begin{align*}
&\partial_x\varphi_1 \partial_{xx} \varphi_1 + \partial_x \varphi_2 \partial_{xx} \varphi_2 + \partial_x \varphi_3 \partial_{xx} \varphi_3 + \partial_x \varphi_4 \partial_{xx} \varphi_4 \\
&= m \Bigg[ (m + 1) \Big\{ (\alpha - \beta) \sin ((m + 1)x) - (\alpha + \beta) \cos ((m + 1)x) \Big\}\Bigg]
\end{align*}
Similarly we get
\begin{align*}
& \varphi_1 \partial_x \varphi_1 + \varphi_2 \partial_x \varphi_2 + \varphi_3 \partial_x \varphi_3 + \varphi_4 \partial_x \varphi_4\\
& = (m + 1) \Big\{ (\alpha + \beta) \cos ((m + 1)x) - (\alpha - \beta) \sin ((m + 1)x) \Big\} 
\end{align*}
So
\begin{align*}
&\eta - \sum _{i=1}^{4} \varphi_{i}\partial_x\varphi_i - (1 - \partial_{xx})^{-1} \sum _{i=1}^{4} \Bigg(2\varphi_i\partial_x\varphi_i + \partial_x \varphi_i \partial_{xx}\varphi_i \Bigg) \\
&= - (m + 1) \Big\{ (\alpha + \beta) \cos ((m + 1)x) - (\alpha - \beta) \sin ((m + 1)x) \Big\} \\
& \quad - (1 - \partial_{xx})^{-1} \Bigg(2 (m + 1) \Big\{ (\alpha + \beta) \cos ((m + 1)x) - (\alpha - \beta) \sin ((m + 1)x) \Big\} \\
&\hspace{3cm} +m \Bigg[ (m + 1) \Big\{ (\alpha - \beta) \sin ((m + 1)x) - (\alpha + \beta) \cos ((m + 1)x) \Big\}\Bigg]  \Bigg)\\
& = - (m + 1) \Big\{ (\alpha + \beta) \cos ((m + 1)x) - (\alpha - \beta) \sin ((m + 1)x) \Big\}\\
& \quad - (1 - \partial_{xx})^{-1} \Bigg( (\alpha - \beta) \Big\{ m (m + 1) - 2(m + 1)  \Big\} \sin ((m + 1)x) \\
& \hspace{5cm}+ (\alpha + \beta) \Big\{ 2(m + 1) -m(m + 1) \Big\}  \cos ((m + 1)x) \Bigg)\\
& = - (m + 1) \Big\{ (\alpha + \beta) \cos ((m + 1)x) - (\alpha - \beta) \sin ((m + 1)x) \Big\} \\
& \quad - \frac{(\alpha - \beta) (m - 2) (m + 1)}{ 1 + (m + 1)^2} \sin ((m + 1)x) - \frac{(\alpha + \beta) (2 - m) (m + 1)}{ 1 + (m + 1)^2} \cos ((m + 1)x) \\
& = (\alpha - \beta) (m + 1) \Bigg\{ 1 - \frac{ (m - 2) }{ 1 + (m + 1)^2} \Bigg\} \sin ((m + 1)x) - (\alpha + \beta) (m + 1) \Bigg\{ 1 
 + \frac{ (2 - m) }{ 1 + (m + 1)^2} \Bigg\} \cos ((m + 1)x)
\end{align*}
Now  choosing once $\alpha = \beta $ then $\alpha = -\beta$ in our chosen $\varphi_i$'s implies $\sin(m + 1)x \text{ , } \cos(m + 1)x \in \mathcal{H}_{m}.$ Then combining the above results, we conclude that \eqref{ind_subspace} holds for $m + 1.$ Then proof is complete. 
\end{proof}

Now we define the following sets

$$\Theta(u_0, t_*) = \left\{ \eta \in L^2_{loc}( R^+; H^{s - 2}(\mathbb{T})) \Big| \text{ solution of \eqref{4} exists and continuous for}\ t \leq t_* \right\}.$$
and 
$$\widehat{\Theta}(u_0, t_*) = \left\{ (\varphi, \eta ) \in  H^{s + 1}(\mathbb{T}) \times L^2_{loc}(R^+; H^{s -2}(\mathbb{T})) \Big| \text{  solution of \eqref{4} exists in  } C([0, t_*]; H^s(\mathbb{T}))\right\}$$

But to prove Theorem \ref{main_thm} we want a finite dimensional control, i.e aim to find $\eta \in \Theta(u_0, T) \cap L^2(0, T; \mathcal{H})$.

The following proposition shows that the nonlinear term of the equation appears in the limit of the solution as $ t $ goes to 0. In other words, we can approximately reach  to the elements of $\mathcal{H}_N.$

\begin{proposition} [Asymptotic property]\label{prp_limit}
Let $ s > \frac{3}{2},$ for all $ u_0, \varphi , \eta_0 \in H^{s + 1}(\mathbb{T})$, then for \eqref{gen_equn} there exists $\delta_0 > 0$ such that $(\delta^{-\frac{1}{2}} \varphi, \delta^{-1} \eta_0) \in \widehat{\Theta}(u_0, t_*)$ for any $\delta \in (0,\delta_0)$, the following limit holds at $t = \delta$
\begin{align*}
\mathcal{R}_{\delta}(u_0, \delta^{-\frac{1}{2}}\varphi, \delta^{-1}\eta_0) \rightarrow u_0 - \varphi\varphi_x + (1 - \partial_{xx})^{-1}\left(\eta_0 - 2\varphi\varphi_x - \varphi_x\varphi_{xx}\right), \text{in}\ H^s(\mathbb{T}) \text{ as }\ \delta \rightarrow 0. 
\end{align*}
\end{proposition}

Since the space  $\mathcal{H}_{\infty} := \bigcup\limits_{n\in\mathbb{N}} \mathcal{H}_{n - 1}$  is dense in $H^s(\mathbb{T}),$  we can deduce from the previous Propositions that  for all $$z = a_0 + \sum_{k=1}^{\infty}a_k cos(kx) + b_k sin(kx) \in H^s(\mathbb{T}),$$ we can find $N ( \varepsilon)$ large enough such that  $\left( a_0 + \sum\limits_{k=1}^{N(\varepsilon)}a_k cos(kx) + b_k sin(kx)\right) \in \mathcal{H}_N $ and $$\left\| z - a_0 + \sum\limits_{k=1}^{N(\varepsilon)}a_k cos(kx) + b_k sin(kx) \right\|_{H^s} < \varepsilon.$$ 

Now the elements of space $\mathcal{H}_n$ formed by trigonometric polynomials are in fact elements of space $\mathcal{F}(\mathcal{H}_n)$ that we have already reach according to proposition \ref{prp_limit}.

\begin{remark}
 As a consequence of Proposition\ref{prp_saturating} and Proposition \ref{prp_limit}, we get approximate controllability by a control in $L^2(0, T; \mathcal{H}_N),$ for some large $N.$ Then an immediate question is : What should be the optimal finite dimensional subspace $\mathcal{H}$ of $L^2(\mathbb{T})$ for which the above holds? Novelty of the Theorem \ref{main_thm} is answering this question by constructing a control in $L^2(0, T; \mathcal{H}),$  where $$\mathcal{H} = \text{span}\{1, \sin(x), \cos(x)  \}.$$
\end{remark}

We finish this section by giving an algebraic property of $\mathcal{R}_{t}$ :
\begin{lemma}\label{lem_rt}
Let $\mathcal{R}_t(u_0, 0, \eta)$ be the solution of \eqref{4}, where $\eta$ is given by $$\eta(s) = \begin{cases}
\eta_1(s), s \in [0, t_1] \\ \eta_2(s), s \in [t_1, t_2]\\\eta_3(s), s \in [t_2, t_3]
\end{cases}$$ For all $t_1, t_2, t_3 \ge 0,$ we have the equality $$ \mathcal{R}_{t_1 +t_2 +t_3}(u_0,0, \eta) = \mathcal{R}_{t_3}(\mathcal{R}_{t_2}(\mathcal{R}_{t_1}(u_0, 0, \eta_1(\cdot)), 0, \eta_2(\cdot -t_1)),0, \eta_3(\cdot -t_2-t_1)).$$
\end{lemma}
\begin{proof}
We denote by $\widehat{\mathcal{R}}(t, s, v, \eta)$ the solution of \eqref{4} at the instant $t,$ when $\varphi = 0$ and with initial data $\widehat{\mathcal{R}}(s, s, v, \eta) = v.$ That means $\mathcal{R}_{t}(u_0, 0, \eta ) = \widehat{\mathcal{R}}(t, 0, u_0, \eta).$ From the uniqueness of the solutions, we can see that for all $\sigma \ge 0$ 
\begin{align}\label{uni_R}
\widehat{\mathcal{R}}(t, \sigma, \widehat{\mathcal{R}}(\sigma, s, u_0, \eta)) = \widehat{\mathcal{R}}(t, s, u_0, \eta)
\end{align} Using\eqref{uni_R} we can write
\begin{align*}
&\widehat{\mathcal{R}}(t_1 + t_2 + t_3, 0, u_0, \eta)\\
& = \widehat{\mathcal{R}}(t_1 + t_2 + t_3, t_1 + t_2,\widehat{\mathcal{R}}(t_1 + t_2, t_1, \widehat{\mathcal{R}}(t_1, 0,  u_0, \eta_1(\cdot)), \eta_2(\cdot - t_1)), \eta_3(\cdot - t_2 -t_1))\\
& = \widehat{\mathcal{R}}(t_3, 0,\widehat{\mathcal{R}}(t_1 + t_2, t_1, \widehat{\mathcal{R}}(t_1, 0,  u_0, \eta_1(\cdot)), \eta_2(\cdot - t_1)), \eta_3(\cdot - t_2 -t_1)) \\
& = \mathcal{R}_{t_3}(\widehat{\mathcal{R}}(t_1 + t_2, t_1, \widehat{\mathcal{R}}(t_1, 0,  u_0, \eta_1(\cdot)), \eta_2(\cdot - t_1)), \eta_3(\cdot - t_2 -t_1)) \\
& = \mathcal{R}_{t_3}(\mathcal{R}_{t_2}(\mathcal{R}_{t_1}(u_0, \eta_1(\cdot)), \eta_2(\cdot - t_1)), \eta_3(\cdot - t_2 -t_1)) 
\end{align*}
\end{proof}

Assuming the Proposition \ref{prp_existance} - \ref{prp_limit} , let us prove Theorem \ref{main_thm}.

\section{Proof of Theorem \ref{main_thm}}
As we have seen the equation is wellposed for $s > \frac{3}{2}$ then through out this section the we will consider $H^s(\mathbb{T})$ for $s > \frac{3}{2},$ what we have discussed in the Introduction, the idea is to establish approximate controllability in small time to the points of the affine space $u_0 + \mathcal{H}_N$ by combining Proposition \ref{prp_limit} and an induction argument in $N.$ Then induction hypothesis as follows
\begin{align*}
\forall u_0 \in H^s(\mathbb{T})\text{ , } \forall N \in \mathbb{N}\text{ , } \forall w \in \mathcal{H}_N \text{ , } \forall \sigma > 0 \text{ , } \exists t \in [0, \sigma],\\ \exists \widehat\eta \in \Theta(u_0, t) \cap L^2(0, t; \mathcal{H}_0) \text{ such that } \|\mathcal{R}_t(u_0, 0, \widehat\eta )  - (u_0 + w)\|_s < \varepsilon    
\end{align*}
 Then the saturation property will imply approximate controllability in small time to any point of $H^s$. Finally, controllability in any time T is proved by steering the system close to the target $u_1$ in small time, then forcing it to remain close to $u_1$ for a sufficiently long time. The accurate proof is divided into four steps.

\vspace{2.5mm}
\textbf{Step 1: Controllability in small time to $u_0 + \mathcal{H}_0$.} Let us assume for the moment that $u_0 \in H^{s+1}$. First we prove that problem \eqref{ctrl_equn} is approximately controllable to the set $u_0+\mathcal{H}_0$ in small time. More precisely, we show that, for any $ \eta \in \mathcal{H}_0, \varepsilon > 0,$ there exists a small time $t > 0$ and a control $\hat{\eta} \in \Theta(u_0, t) \cap L^2(0, t; \mathcal{H}_0) $ such that $$\|\mathcal{R}_t(u_{0}, 0,\hat{\eta}) - (u_{0} + \eta)\|_{s} < \epsilon.$$ Indeed, applying Proposition \ref{prp_limit} for the couple $(\eta, 0)$, we see that $$\mathcal{R}_{\delta}(u_{0}, 0, \delta^{-1}\eta) \to u_{0} + \eta\ \hspace{.5cm} \text{in}\  H^s(\mathbb{T})\ \hspace{0.5cm} \text{as}\ \delta \to 0.$$ Which gives the required result with $\hat{\eta} = t^{-1}(1 - \partial_{xx})\eta $ and $t = \delta.$

\vspace{2.5mm}
\textbf{Step 2. Controllability in small time to $u_{0} + \mathcal{H}_{N}$.} After getting approximate controllability to $u_0 + \mathcal{H}_0$ how can we reach very close to $u_0 + \mathcal{H}_1$ we have discussed it explicitly for a simple case (See Appendix), for the time being to prove the general case, we will use induction, Assume that approximate controllability of  the control problem \eqref{ctrl_equn} to the set $u_0 + \mathcal{H}_{N-1}$ is already proved. Let $\Tilde{\eta} \in \mathcal{H}_N$ be of the form $$\Tilde{\eta} = \eta - \sum _{i=1}^{m} \varphi_{i}\partial_x\varphi_i - (1 - \partial_{xx})^{-1} \sum _{i=1}^{m} \big(2\varphi_i\partial_x\varphi_i + \partial_x \varphi_i \partial_{xx}\varphi_i \big) $$ for some $m \ge 1,$ and the vectors $\eta, \varphi_1, ...... , \varphi_m \in \mathcal{H}_{N-1}.$ Applying the Proposition \ref{prp_limit}, we see that there exists $\theta_1 > 0,$ and control $\eta_1 \in \Theta(u_0, \theta_1) \cap L^2(0, \theta_1; \mathcal{H}_0)$ such that 
\begin{align}\label{ineq_1}
\|\mathcal{R}_{\theta_1}(u_0, 0, \eta_1) - (u_0 + \theta_1^{-\frac{1}{2}} \varphi_1) \|_s < \frac{\varepsilon}{2}.
\end{align}
By the uniqueness of the solution of the Cauchy problem, the following equality holds $$\mathcal{R}_t(u_0 + \delta^{-\frac{1}{2}}v, 0, \delta^{-1}w) = \mathcal{R}_t(u_0, \delta^{-\frac{1}{2}}v, \delta^{-1}w) + \delta^{-\frac{1}{2}}v \text{ ,  for all } t \in [0, t_*(\delta)]$$
Combining this with the fact that $\eta, \varphi_1 \in \mathcal{H}_{N -1},$ induction hypothesis and Proposition \ref{prp_limit}, we can find a small time $\theta_2 > 0,$ and $\eta_2  \in \Theta(u_0, \theta_2) \cap L^2(0, \theta_2 ; \mathcal{H}_0) $ such that
\begin{align}\label{ineq_2}
\|\mathcal{R}_{\theta_2}(u_0 + \theta_1^{-\frac{1}{2}} \varphi_1, 0, \eta_2) - (u_0 + \eta - \varphi_1 \partial_x \varphi_1 - (1 - \partial_{xx})^{-1}\big(2\varphi_1\partial_x\varphi_1 + \partial_x \varphi_1 \partial_{xx}\varphi_1 \big))\|_s < \frac{\varepsilon}{2}.
\end{align}
Now define the control $\widehat\eta_1 : s \to \mathbbm{1}_{[0, \theta_1]} \eta_1 + \mathbbm{1}_{[\theta_1 ,  \theta_1 + \theta_2]} \eta_2 $ and using the Lemma \ref{lem_rt} and Equation \eqref{ineq_1}  \eqref{ineq_2} we have 
\begin{align}\label{ineq_3}
&\|\mathcal{R}_{\theta_1 + \theta_2}(u_0, 0, \widehat\eta_1) - (u_0 + \eta - \varphi_1 \partial_x \varphi_1 - (1 - \partial_{xx})^{-1}\big(2\varphi_1\partial_x\varphi_1 + \partial_x \varphi_1 \partial_{xx}\varphi_1 \big))  \|_s \notag \\
&\le \|\mathcal{R}_{\theta_2}(\mathcal{R}_{\theta_1}(u_0, 0, \mathbbm{1}_{[0, \theta_1]} \eta_1), 0, \eta_2) -  \mathcal{R}_{\theta_2}(u_0 + \theta_1^{-\frac{1}{2}} \varphi_1, 0, \eta_2)\|_s \notag \\
& \hspace{1.5cm}+ \| \mathcal{R}_{\theta_2}(u_0 + \theta_1^{-\frac{1}{2}} \varphi_1, 0, \eta_2) - (u_0 + \eta - \varphi_1 \partial_x \varphi_1 - (1 - \partial_{xx})^{-1}\big(2\varphi_1\partial_x\varphi_1 + \partial_x \varphi_1 \partial_{xx}\varphi_1 \big)) \|_s < \varepsilon.
\end{align}
Following the method above with minor changes, for initial data $\widehat u_0 = u_0 + \eta - \varphi_1 \partial_x \varphi_1 - (1 - \partial_{xx})^{-1}\big(2\varphi_1\partial_x\varphi_1 + \partial_x \varphi_1 \partial_{xx}\varphi_1 \big) \in H^{s +1}(\mathbb{T}),$ there exists a small time $\theta_3 > 0$ and a control $\eta_3 \in \Theta(u_0, \theta_3) \cap L^2(0, \theta_3; \mathcal{H}_0)$ such that
\begin{align}\label{ineq_4}
\|\mathcal{R}_{\theta_3}(\widehat u_0, 0,  \eta_3) - (\widehat u_0 - \varphi_2 \partial_x \varphi_2 - (1 - \partial_{xx})^{-1}\big(2\varphi_2\partial_x\varphi_2 + \partial_x \varphi_2 \partial_{xx}\varphi_2 \big))\|_s < \varepsilon.
\end{align}
This means starting form $u_0 + \eta - \varphi_1 \partial_x \varphi_1 - (1 - \partial_{xx})^{-1}\big(2\varphi_1\partial_x\varphi_1 + \partial_x \varphi_1 \partial_{xx}\varphi_1 \big) ,$ we can attain approximately $u_0 + \eta - \varphi_1 \partial_x \varphi_1 - (1 - \partial_{xx})^{-1}\big(2\varphi_1\partial_x\varphi_1 + \partial_x \varphi_1 \partial_{xx}\varphi_1 \big) - \varphi_2 \partial_x \varphi_2 - (1 - \partial_{xx})^{-1}\big(2\varphi_2\partial_x\varphi_2 + \partial_x \varphi_2 \partial_{xx}\varphi_2 \big) .$ Now taking $\widehat\eta_2: s \to \mathbbm{1}_{[0, \theta_1 + \theta_2]} \widehat\eta_1 + \mathbbm{1}_{[\theta_1+\theta_2, \theta_1 + \theta_2 + \theta_3]}\eta_3$ as a control and combining Lemma \ref{lem_rt} and Equation \eqref{ineq_3}, \eqref{ineq_4} we have
\begin{align}\label{ineq_5}
\|\mathcal{R}_{\theta_1 + \theta_2 + \theta_3}(u_0, 0, \widehat\eta_2) - (u_0 + \eta - \sum _{i=1}^{2} \varphi_{i}\partial_x\varphi_i - (1 - \partial_{xx})^{-1} \sum _{i=1}^{2} \big(2\varphi_i\partial_x\varphi_i + \partial_x \varphi_i \partial_{xx}\varphi_i \big))\|_s   < \varepsilon.  
\end{align}
Choose $\theta_1 , \theta_2, \theta_3$ such that $\theta_1 + \theta_2 + \theta_3 < \sigma.$

Iterating the argument, we construct a small time $ \theta > 0,$ and a control $\widehat\eta \in L^2(0, \theta; \mathcal{H}_0)$ satisfying 
\begin{align}\label{ineq_6}
\|\mathcal{R}_{\theta}(u_0, 0, \widehat\eta) - (u_0 + \eta - \sum _{i=1}^{m} \varphi_{i}\partial_x\varphi_i - (1 - \partial_{xx})^{-1} \sum _{i=1}^{m} \big(2\varphi_i\partial_x\varphi_i + \partial_x \varphi_i \partial_{xx}\varphi_i \big))\|_s  \notag \\
= \|\mathcal{R}_{\theta}(u_0, 0, \widehat\eta) - (u_0 + \Tilde{\eta})\|_s < \varepsilon.
\end{align}
This proves approximate controllability in small time to any point in $u_0 + \mathcal{H}_N.$

\vspace{2.5mm}
\textbf{Step 3. Global controllability in small time.} Now let $u_{1} \in H^s(\mathbb{T})$ be arbitrary. As $\mathcal{H}_{\infty}$ dense in $H^s(\mathbb{T})$, there is an integer $N \ge 1$ and point $\hat{u}_{1} \in u_{0} + \mathcal{H}_N$ such that 
\begin{equation}\label{ineq_7}
\|u_{1} - \hat{u}_{1}\|_{s} < \frac{\varepsilon}{2}
\end{equation}
By the results of steps $1$ and $2$, for any $\varepsilon > 0$, there is a time $\theta > 0$ and a control $\hat{\eta} \in L^2(0,T;\mathcal{H})$ satisfying $$\|\mathcal{R}_{\theta}(u_{0}, 0, \hat{\eta}) - \hat{u}_{1}\|_{s} < \frac{\varepsilon}{2}.$$
Combining this with \eqref{ineq_7}, we get approximate controllability in small time to $u_{1}$ from $u_{0} \in H^{s+1}(\mathbb{T}).$ Since the space $H^{s + 1}(\mathbb{T})$ is dense in $H^s(\mathbb{T})$ and proposition \ref{prp_cntity} we conclude small time approximate controllability starting from arbitrary $u_{0} \in H^s(\mathbb{T})$.

\vspace{2.5mm}
\textbf{Step 4. Global approximate Controllability in fixed time $T$.} Since we have goal controllability in small time, to complete the proof of the theorem, it suffices to show that, for any $\varepsilon, T > 0$ and any $u_1 \in H^s,$ there is a control $ \eta \in \Theta(u_1, T) \cap L^2(0, T; \mathcal{H}_0)$ such that 
\begin{align}\label{ineq_8}
\|\mathcal{R}_{T}(u_1, 0, \eta) - u_1\|_s < \varepsilon.
\end{align}
Note that here the initial condition and the target coincide with $u_1.$It is not clear, whether it is possible or not to find a control taking values in $ \mathcal{H}_0$ such that the solution starting from $ u_1$ remains close to that point on all the time interval $[0, T]$ . However, we will see it is possible . 

So applying the result of step $3,$ for any $\varepsilon > 0$, there is a time $T_{1} > 0$ and a control $\widehat\eta^1 \in L^2(0,T_1;\mathcal{H}_0)$ satisfying $$\|\mathcal{R}_{T_{1}}(u_{0}, 0, \widehat\eta^1) - u_{1}\|_{s} < \frac{\varepsilon}{2}.$$

Take $v_{1} = \mathcal{R}_{T_{1}}(u_{0}, 0, \widehat\eta^1)$. According to Proposition \ref{prp_cntity}, we can find $\tau > 0$ such that for $t \in [0,\tau]$,$$\|\mathcal{R}_{t}(v_{1}, 0, 0)- v_{1}\|_{s} < \frac{\varepsilon}{2}.$$ Define a control function $$\overline{\eta}_{1}(t)= \begin{cases}
     \widehat\eta^1(t)\hspace{1cm} t \in [0,T_{1}]\\0 \hspace{1.5cm} t \in (T_{1}, T_{1}+\tau],
 \end{cases}$$
then, it follows that $$\|\mathcal{R}_{T_{1} +t}(u_{0}, 0,  \overline{\eta}_{1}) - u_{1}\|_{s} < \epsilon,\hspace{1cm} \forall t \in [0, \tau].$$
If $T_{1} + \tau \ge T$, then the proof is complete. Otherwise , take $v_{2} = \mathcal{R}_{T_{1} + \tau}(u_{0}, 0, \overline{\eta}_{1})$, by the result of step $3,$ there is a time $T_{2} > 0$ and a control $\widehat\eta^{2} \in L^2(0, T_{2};\mathcal{H}_0)$ satisfying $$\|\mathcal{R}_{T_{2}}(v_{2}, 0,  \widehat \eta^{2}) - u_{1}\|_{s} < \frac{\epsilon}{2}.$$
 Take $v_{3} = \mathcal{R}_{T_{2}}(v_{2}, 0,  \widehat\eta^{2})$. According to Proposition \ref{prp_existance},  for the same $\tau $, if $t \in [0,\tau]$, we have $$\|\mathcal{R}_{t}(v_{3}, 0, 0) - v_{3}\|_{s} < \frac{\varepsilon}{2}.$$ 
Define a control function $$\overline{\eta}_{2}(t)= \begin{cases}
     \overline{\eta}_{1}(t)\hspace{1cm} t \in [0,T_{1} + \tau]\\ \widehat\eta^{2}(t)\hspace{1cm} t \in (T_{1} + \tau,T_{1} + T_{2} + \tau]\\0 \hspace{1.5cm} t \in (T_{1} + T_{2} + \tau, T_{1} + T_{2} + 2\tau],
 \end{cases}$$

Then by the lemma \ref{lem_rt}, we have $$\|\mathcal{R}_{T_{1}+T_{2}+\tau+t}(u_{0}, 0, \overline{\eta}_{2}) - u_{1}\|_{s} < \varepsilon,\hspace{1cm} \forall t \in [0, \tau].$$

Again if $T_{1} + T_{2} + 2\tau \ge T$, then the proof is complete. Other wise, we apply small time controllability property to return to the ball $B_{H^s}(u_1, r)$ for any numbers $r \in (0, \frac{\varepsilon}{2}),$ after a finite number (less than the integer part of $\frac{T}{\tau+1})$ of iterations, we complete the proof of Theorem \ref{main_thm}.

\section{Proof of the Propositions}
First we introduce some notations, as we have defined $\Lambda^s := (1 - \partial_{xx})^{\frac{s}{2}}$, the pseudo-differential operator $\Lambda^s $ is defined for any $s \in \mathbb{R}$ on a test function $f$ by $$\widehat{\Lambda^s f}(\xi) = (1 + \xi^2)^{\frac{s}{2}} \hat{f}(\xi), $$ where $\hat{f}$ denotes the Fourier transformation of a function $f$ on the circle $\mathbb{T} = \mathbb{R}/2\pi \mathbb{Z}$, for $\xi \in \mathbb{Z}$ $$\hat{f}(\xi) = \int_{\mathbb{T}} e^{-i\xi x} f(x) dx . $$ 
Also, we recall that for any $s \in \mathbb{R}$ the sobolev space $H^s = H^s(\mathbb{T})$ is defined by $$H^s(\mathbb{T}) = \left\{ f \in \mathcal{D}'(\mathbb{T}) : \|f\|_{H^s} = \|\Lambda^s f\|_{L^2} \simeq \left(\sum_{\xi \in \mathbb{Z}} \left(1 + \xi^2\right)^{s} |\hat{f}(\xi)|^2\right)^{\frac{1}{2}} < \infty \right\}.$$
Furthermore, the map $\Lambda^s : H^r \to H^{r - s}$ has the operator norm
\begin{align}\label{lambda_ineq}
\|\Lambda^s\|_{\mathcal{L}(H^r, H^{r-s})} = 1 \iff \|\Lambda^s f\|_{H^{r-s}} \le \|f\|_{H^r}   \textit{ ,  } \forall f \in H^r.
\end{align}
As an operator between Sobolev spaces, we will use the fact that $\partial_x : H^r \to H^{r - 1}$ satisfies 
\begin{align}\label{dx_ineq}
\|\partial_x\|_{\mathcal{L}(H^r, H^{r-1})} = 1 \iff \|\partial_x f\|_{H^{r - 1}} \le \|f\|_{H^r} \textit{  ,  } \forall f \in H^r.
\end{align}
We adopt the notation $P \lesssim Q$ for the positive quantities $P$ and $ Q $ if there exists a constant $c > 0 $ such that $P \le c Q.$

Next, we collect some properties of the pseudo-differential operator $\Lambda^s$ and the $H^s$ space which will be used.
\begin{lemma}\label{lemma_commutator} 
As defined $H^s$ and $\Lambda^s$ for $s > 0,$ we have the followings
\begin{enumerate}
\item $ H^s $ forms an algebra for $s > \frac{1}{2}$, so the following holds
\begin{align}\label{alg_Hs}
\|fg\|_{H^s} \lesssim \|f\|_{H^s} \|g\|_{H^s} \textit{ , } \hspace{1cm} \forall f , g \in H^s.
\end{align}
\item If $s > 0$ then there is $c_s > 0,$ such that
\begin{align}\label{Kato_Ponce_estimate}
\|[\Lambda^s , f] g\|_{L^2} \le c_s \Big( \|\Lambda^s f\|_{L^2} \|g\|_{L^{\infty}}  + \|\partial_x f\|_{L^{\infty}} \|\Lambda^{ s - 1} g\|_{L^2} \Big) .   
\end{align}
where $[\Lambda^s , f] = \Lambda^s f - f \Lambda^s$ is the commutator, in which $ f $ is regarded as a multiplication operator and $ [\Lambda^s , f] g= \Lambda^s(fg) - f \Lambda^s(g).  $
\end{enumerate}
\end{lemma}
For the details proof of the above lemma see (appendix \cite{Kato_Ponce_1988}).

Now we can prove Proposition \ref{prp_existance}.

\vspace{5mm}
\subsection{Proof of Proposition \ref{prp_existance}} For given $u_0  \in H^s$ and $ \varphi, f$ in some suitable space, which will be decided later. we can write \eqref{inverse_equn} as
\begin{equation}\label{inv_eqn_main}
\begin{cases}
u_t = -(u+\varphi)\partial_x(u+\varphi)-F(u)\\
u(0,x)=u_0(x).
\end{cases}
\end{equation}
where
\begin{align}\label{exp_F(u)}
F(u) = \Lambda^{-2}\Big[ \partial_x(u+\varphi) + 2(u+\varphi)\partial_x (u+\varphi) + \partial_x(u+\varphi)\partial_{xx}(u+\varphi) + f\Big] 
\end{align}
To show the existences of \eqref{inv_eqn_main} we will treat the equation like an IVP in the Banach Space $H^s.$ But there is a problem regarding the term $u u_x.$ So using a Friedrichs mollifier $J_{\varepsilon},$ we obtain the following mollified version of the  Cauchy problem \eqref{inv_eqn_main} $u_t = - [(J_{\varepsilon}u) (J_{\varepsilon}u_x) + u\varphi_x + \varphi u_x + \varphi \varphi_x] - F(u), u(0, x) = u_0(x),$ which is genuine ODE problem in $H^s$ and which can be solved using the abstract ODE result (See Theorem $7.3$ Chapter $7$,\cite{Functional_Brezis_book}). Using energy estimates, it is shown that this solution is unique. For this part we will follow \cite{Himonas_2014}.

\textbf{The Mollified i.v.p.} Next , we study the following mollified version of problem  of \eqref{inv_eqn_main}
\begin{align}\label{mollified_ode}
\begin{cases}
u_t = - [(J_{\varepsilon} u) (J_{\varepsilon} u_x)  + u\varphi_x + \varphi u_x + \varphi \varphi_x] - F(u) \\
u(0, x) = u_0(x) 
\end{cases}
\end{align}
where for each $\varepsilon \in (0, 1]$ the operator $J_{\varepsilon}$ is the Friedrichs mollifier, defined by 
\begin{align}
J_{\varepsilon} f : = j_{\varepsilon} * f .
\end{align}
 To define $j_{\varepsilon},$ we fix a Schwartz function $j(x) \in \mathcal{S}(\mathbb{R})$ satisfying $ 0 \le \widehat{j}(\xi) \le 1$ for all $\xi \in \mathbb{R}$ and $\widehat{j}(\xi) = 1$ for $|\xi | \le 1.$ We may then define the periodic function $j_{\varepsilon}$ by
 \begin{align}
j_{\varepsilon}(x) := \frac{1}{2\pi} \sum_{n \in \mathbb{Z}} \widehat{j}(\varepsilon n) e^{i n x}.
\end{align}
From the construction of mollifier  $j_{\varepsilon},$ we have 
\begin{align}\label{commutative}
\Lambda^s J_{\varepsilon} = J_{\varepsilon} \Lambda^s 
\end{align} and
\begin{align}\label{self_adj}
\langle J_{\varepsilon} f , g  \rangle_{L^2} = \langle J_{\varepsilon} g , f  \rangle_{L^2}
\end{align}
We now consider the map $G_{\varepsilon} : H^s \to H^s,$ given by
\begin{align}\label{G_epsilon_map}
  G_{\varepsilon}(u) =  - [(J_{\varepsilon} u) (J_{\varepsilon} u_x) + u\varphi_x + \varphi u_x + \varphi \varphi_x] - F(u).
\end{align}
Each map $G_{\varepsilon} $ is continuously differentiable.
Consider 
\begin{align}\label{engy_est}
&\frac{d}{dt}  \|J_{\epsilon} u\|_{H^s}^2  =  \frac{d}{dt}\langle \Lambda^s J_{\epsilon}u , \Lambda^s J_{\epsilon}u \rangle_{L^2} = 2 \langle \Lambda^s \partial_t J_{\epsilon}u , \Lambda^s J_{\epsilon}u \rangle_{L^2}\notag \\ 
& \quad = - 2 \langle \Lambda^s  J_{\epsilon} [(J_{\varepsilon} u) (J_{\varepsilon} u_x)] , \Lambda^s J_{\epsilon}u \rangle_{L^2} -2 \langle \Lambda^s  J_{\epsilon} [u\varphi_x + \varphi u_x + \varphi \varphi_x ] , \Lambda^s J_{\epsilon}u \rangle_{L^2} - 2 \langle \Lambda^s  J_{\epsilon} F(u) , \Lambda^s J_{\epsilon}u \rangle_{L^2} 
\end{align}
We now rewrite the first term of \eqref{engy_est} by first commuting the $J_{\varepsilon}$ and then using \eqref{self_adj}, arriving at
\begin{align}\label{engy_est_1}
&\left|\int_{\mathbb{T}} \Lambda^s [(J_{\varepsilon} u) (J_{\varepsilon} u_x)] \textit{ . } \Lambda^s J^2_{\epsilon}u dx\right| = \Bigg|\int_\mathbb{{T}} [\Lambda^s , J_{\varepsilon} u] \partial_{x} J_{\varepsilon} u \textit{ }\Lambda^s J_{\varepsilon}^2 u dx \notag \\
& \hspace{4cm}+ \int_{\mathbb{T}} J_{\varepsilon} u \partial_{x} \Lambda^s J_{\varepsilon} u \textit{ } \Lambda^s J_{\varepsilon}^2 u dx\Bigg|
\end{align}
where we have added and subtracted $J_{\varepsilon} u \partial_x \Lambda^s J_{\varepsilon} u $ and used the commutator.  Setting $ v = J_{\varepsilon} u,$ we can bound the first term of \eqref{engy_est_1} by first using Cauchy-Schwarz inequality  to arrive at 
\begin{align}\label{est_1_a}
\left|\int_{\mathbb{T}} [\Lambda^s, v] \partial_{x} v \textit{ } \Lambda^s J_{\varepsilon} v dx\right| \le \Big\|[\Lambda^s, v] \partial_{x} v\Big\|_{L^2} \Big\|\Lambda^s J_{\varepsilon} v\Big\|_{L^2}
\end{align}
Now using $\|J_{\varepsilon} u\|_{H^s} \le \|u\|_{H^s},$ and definition of $H^s$ norm we get $\Big\|\Lambda^s J_{\varepsilon} v\Big\|_{L^2} \le \|u\|_{H^s}.$ Then applying part $(2)$ of Lemma \ref{lemma_commutator} on $ \Big\|[\Lambda^s, v] \partial_{x} v\Big\|_{L^2} $  from \eqref{est_1_a} we get
\begin{align}
&\left|\int_{\mathbb{T}} [\Lambda^s, v] \partial_{x} v \textit{ } \Lambda^s J_{\varepsilon} v dx \right|\le \Big( \|\Lambda^s v\|_{L^2} \|\partial_x v\|_{L^{\infty}} + \|\partial_x v\|_{L^{\infty}} \|\partial_x v\|_{H^{s - 1}} \Big) \|u\|_{H^s} \notag \\
& \hspace{ 4cm }\le \Big( \|v\|_{H^s} \|\partial_x v\|_{L^{\infty}} + \|\partial_x v\|_{L^{\infty}} \|v\|_{H^s} \Big) \|u\|_{H^s}
\end{align}
Finally we get 
\begin{align}\label{est_1_a_fnl}
&\left|\int_{\mathbb{T}} [\Lambda^s, v] \partial_{x} v \textit{ } \Lambda^s J_{\varepsilon} v dx\right| \le \Big(\|J_{\varepsilon} u\|_{H^s} \|u_x\|_{L^{\infty}} \Big) \|u\|_{H^s} \ \textit{ } \Bigg( As \|\partial_xJ_{\varepsilon}u\|_{L^{\infty}} \le \|u_x\|_{L^{\infty}}\Bigg) \notag \\
& \hspace{4cm} \le \|u\|^2_{H^s} \|u_x\|_{L^{\infty}}
\end{align}
Now consider the second term of \eqref{engy_est_1} with setting $v = J_{\varepsilon}u$ we get,
\begin{align}\label{est_1_b}
\left|\int_{\mathbb{T}} v \partial_x \Lambda^s v J_{\varepsilon} \Lambda^s v dx\right| & = \left|\int_{\mathbb{T}} J_{\varepsilon} \Big( v \partial_x \Lambda^s v \Big) \Lambda^s v dx\right| \hspace{1cm} ( As \textit{ } \Lambda^s J_{\varepsilon} = J_{\varepsilon} \Lambda^s) \notag \\
&\quad =\left| \int_{\mathbb{T}} \Big([J_{\varepsilon}, v] \Lambda^s v_{x} \Big) \Lambda^s v dx +  \int_{\mathbb{T}}  v \Big( J_{\varepsilon} \partial_x \Lambda^s v \Big) \Lambda^s v dx\right| \notag \\
&\quad = \left|\int_{\mathbb{T}} \Big([J_{\varepsilon}, v] \Lambda^s v_{x} \Big) \Lambda^s v dx - \int_{\mathbb{T}} \partial_x v \Big( J_{\varepsilon} \Lambda^s v \Big) \Lambda^s v dx - \int_{\mathbb{T}} v \Big( J_{\varepsilon} \Lambda^s v \Big) \partial_x \Lambda^s v dx\right|
\end{align}
So from \eqref{est_1_b} we have 
\begin{align}\label{est_1_b_fnl}
\left|\int_{\mathbb{T}} v \partial_x \Lambda^s v J_{\varepsilon} \Lambda^s v dx \right| & \le \left|\frac{1}{2} \int_{\mathbb{T}} \Big([J_{\varepsilon}, v] \Lambda^s v_{x} \Big) \Lambda^s v dx\right| + \left|\frac{1}{2} \int_{\mathbb{T}} \partial_x v \Big( J_{\varepsilon} \Lambda^s v \Big) \Lambda^s v dx \right| \notag \\
& \lesssim \|[J_{\varepsilon}, v] \Lambda^s v_{x}  \|_{L^2} \|\Lambda^s v\|_{L^2} + \|\partial_x v\|_{L^{\infty}} \|J_{\varepsilon} \Lambda^s v \|_{L^2} \|\Lambda^s v\|_{L^2} \notag \\
&  \le \Big( \| \partial_x v\|_{L^{\infty}} \|\Lambda^s v\|_{L^2} \Big) \|v\|_{H^s} + \|\partial_x v\|_{L^{\infty}} \|v\|_{H^s}^2 \notag \\
& \lesssim \|u_x\|_{L^{\infty}} \|u\|^2_{H^s}
\end{align}
 where, for the estimate of the first integral of \eqref{est_1_b_fnl}, we used the following lemma applied with $w = v$ and $f = \Lambda^s v.$ Here also for applying next Lemma we have used $ s > \frac{3}{2}.$
\begin{lemma}\label{lemma_L2_L_infty}
Let $ w $ be such that $\|\partial_x w\|_{L^{\infty}} < \infty. $ Then there is a constant $ C > 0 $ such that for any $ f \in L^2, $ we have 
\begin{align}
\| [J_{\varepsilon}, w] \partial_x f \|_{L^2} \le c \|\partial_x w\|_{L^{\infty}} \|f\|_{L^2}.
\end{align}
\end{lemma}
Now consider the second term of \eqref{engy_est}
\begin{align}\label{engy_est_2}
\left|\int_{\mathbb{T}}\Lambda^s  J_{\epsilon} [u\varphi_x + \varphi u_x + \varphi \varphi_x ] \textit{ . } \Lambda^s J_{\epsilon}u dx \right|    &\le \left|\int_{\mathbb{T}}\Lambda^s  J_{\epsilon} u\varphi_x \textit{ . } \Lambda^s J_{\epsilon}u dx \right| + \left|\int_{\mathbb{T}}\Lambda^s  J_{\epsilon} u_x\varphi \textit{ . } \Lambda^s J_{\epsilon}u dx \right| \notag \\
&\hspace{2cm}+ \left|\int_{\mathbb{T}}\Lambda^s  J_{\epsilon} \varphi\varphi_x \textit{ . } \Lambda^s J_{\epsilon}u dx \right|
\end{align}
rewrite the first term
\begin{align}\label{est_2_a}
\left|\int_{\mathbb{T}}\Lambda^s u\varphi_x \textit{ . } \Lambda^s J_{\epsilon}^2u dx \right|  \le \|\Lambda^s u\varphi_x\|_{L^2} \|\Lambda^s J_{\epsilon}^2u\|_{L^2}  \le \| u\varphi_x\|_{H^s} \|u\|_{H^s}  \le \|\varphi_x\|_{H^s} \|u\|_{H^s}^2
\end{align}
where we have first applied Cauchy Schwarz then in the last inequality we have used that $H^s$ is an algebra for $s > \frac{1}{2}.$ Now we know by Young's inequality  for $1 < p < \infty, ab \le \frac{a^p}{p} + \frac{b^q}{q}$ where $\frac{1}{p} + \frac{1}{q} = 1.$Applying Young's inequality taking $ a = \|\varphi_x\|_{H^s}, b= \|u\|_{H^s}^2 $ and $ p = 3, q = \frac{3}{2}$ we have 
\begin{align}\label{est_2_a_fnl}
\left|\int_{\mathbb{T}}\Lambda^s u\varphi_x \textit{ . } \Lambda^s J_{\epsilon}^2u dx \right|  \le \frac{\|\varphi_x\|_{H^s}^3}{3} + \frac{(\|u\|_{H^s}^2)^{\frac{3}{2}}}{\frac{3}{2}} = \frac{\|\varphi_x\|_{H^s}^3}{3} + \frac{\|u\|_{H^s}^3}{\frac{3}{2}}  
\end{align}
Consider the second term of \eqref{engy_est_2} and rewrite it as we have done in \eqref{est_1_b}
\begin{align}\label{est_2_b}
\left|\int_{\mathbb{T}}\Lambda^s  u_x\varphi \textit{ . } \Lambda^s J_{\epsilon}^2u dx \right|  & \le  \left|\frac{1}{2} \int_{\mathbb{T}} \Big([J_{\varepsilon}, \varphi] \Lambda^s u_{x} \Big) \Lambda^s u dx\right| + \left|\frac{1}{2} \int_{\mathbb{T}} \partial_x \varphi \Big( J_{\varepsilon} \Lambda^s u \Big) \Lambda^s u dx \right| \notag \\
&\quad \le \|[J_{\varepsilon}, \varphi] \Lambda^s u_{x}\|_{L^2} \|\Lambda^s u\|_{L^2} + \|\varphi_x\|_{L^{\infty}} \|\Lambda^s u\|_{L^2}^2 \notag \\
& \quad \quad \le \|\varphi_x\|_{L^{\infty}} \|u\|_{H^s}^2
\end{align}
where in the last line we have again used Lemma \ref{lemma_L2_L_infty}.
Now consider the last term of \eqref{engy_est_2}
\begin{align}\label{est_2_c}
\left|\int_{\mathbb{T}}\Lambda^s  J_{\epsilon} \varphi\varphi_x \textit{ . } \Lambda^s J_{\epsilon}u dx \right| \le \|\Lambda^s  J_{\epsilon} \varphi\varphi_x\|_{L^2} \|\Lambda^s J_{\epsilon}u\|_{L^2} \le \|\varphi\varphi_x\|_{H^s} \|u\|_{H^s} \le \|\varphi\|_{H^s} \|\varphi_x\|_{H^s}\|u\|_{H^s}
\end{align}
where we have used for  $s > \frac{1}{2}, H^s$ forms an algebra. Now consider the third term of \eqref{engy_est}
\begin{align}\label{engy_est_3}
\left|\int_{\mathbb{T}} \Lambda^s J_{\varepsilon} F(u) \textit{ . } \Lambda^s J_{\varepsilon} u dx  \right| \le \|F(u)\|_{H^s} \|u\|_{H^s}   
\end{align}
Now  expression of $F(u)$ given by \eqref{exp_F(u)}, so calculate $\|F(u)\|_{H^s}$
\begin{align}\label{nrm_F(u)}
&\left\|\Lambda^{-2}\Big[ \partial_x(u+\varphi) + 2(u+\varphi)\partial_x (u+\varphi) + \partial_x(u+\varphi)\partial_{xx}(u+\varphi) + f\Big]  \right\|_{H^s} \notag \\
& \lesssim \|\Lambda^{-2}\partial_x(u+\varphi)\|_{H^s} + \|\Lambda^{-2} (u+\varphi)\partial_x (u+\varphi)\|_{H^s} + \|\Lambda^{-2} \partial_x(u+\varphi)\partial_{xx}(u+\varphi) \|_{H^s} + \|\Lambda^{-2} f\|_{H^s} \notag \\
& \lesssim \|\Lambda^{-2}\partial_x u \|_{H^s} + \|\Lambda^{-2}\partial_x \varphi \|_{H^s} + \|\Lambda^{-2}\partial_x u^2 \|_{H^s} + \|\Lambda^{-2}\partial_x \varphi^2 \|_{H^s} + \|\Lambda^{-2}\partial_x (u \varphi) \|_{H^s} + \|\Lambda^{-2}\partial_x u_x^2 \|_{H^s} \notag \\
& \hspace{4cm}+ \|\Lambda^{-2}\partial_x \varphi_x^2 \|_{H^s} \|\Lambda^{-2}\partial_x (u_x \varphi_x) \|_{H^s} + \|\Lambda^{-2} f\|_{H^s} \notag \\
& \lesssim \|u\|_{H^{s-1}} + \|\varphi\|_{H^{s-1}} + \|u^2\|_{H^{s-1}} + \|\varphi^2\|_{H^{s-1}} + \|u\varphi\|_{H^{s-1}} + \|u_x^2\|_{H^{s-1}} + \|\varphi_x^2\|_{H^{s-1}} \notag \\ 
& \hspace{6cm}+ \|u_x \varphi_x\|_{H^{s-1}} + \|f\|_{H^{s - 2}} \notag \\
& \le  \|u\|_{H^s} + \|\varphi\|_{H^s} + \|u\|_{H^{s-1}} \|u\|_{H^{s-1}} + \|\varphi\|_{H^{s-1}} \|\varphi\|_{H^{s-1}} + \|u\|_{H^{s-1}} \|\varphi\|_{H^{s-1}}\notag \\
& \quad \hspace{2cm}+ \|u_x\|_{H^{s-1}} \|u_x\|_{H^{s-1}} + \|\varphi_x\|_{H^{s-1}} \|\varphi_x\|_{H^{s-1}} + \|u_x\|_{H^{s-1}} \|\varphi_x\|_{H^{s-1}} + \|f\|_{H^{s-2}}
\end{align}
where we have used that $\Lambda^{-2} \partial_x : H^{s - 1} \to H^s$ is a bounded operator, $ \|u\|_{H^{s-1}} \le \|u\|_{H^s}$ and assuming $s > \frac{3}{2}$ then $H^{s - 1}$ forms an algebra . i.e $$\|u v \|_{H^{s - 1}} \le \|u\|_{H^{s-1}} \|v\|_{H^{s-1}}.$$ So from \eqref{engy_est_3} and \eqref{nrm_F(u)} we get 
\begin{align}\label{engy_est_3_fnl}
\left|\int_{\mathbb{T}} \Lambda^s J_{\varepsilon} F(u) \textit{ . } \Lambda^s J_{\varepsilon} u dx  \right| & \le \Big(\|u\|_{H^s} + \|\varphi\|_{H^s} + \|u\|_{H^s}^2 + \|\varphi\|_{H^s}^2 + \|u\|_{H^s} \|\varphi\|_{H^s} \notag \\
& \hspace{3cm} + \|u\|_{H^s}^2 + \|\varphi\|_{H^s}^2 + \|f\|_{H^{s - 2}}  \Big) \|u\|_{H^s}   
\end{align}
Now assuming $s > \frac{3}{2}, \|\varphi_x\|_{H^s} < \infty$ and $\|f\|_{H^{s - 2}} < \infty$ then using $\|u_x\|_{L^{\infty}} \le \|u\|_{H^s}$ and Young's inequality  from \eqref{engy_est}  taking $ \varepsilon \to 0,$ we get
\begin{align}\label{engy_est_fnl}
& \frac{d}{dt}  \| u\|_{H^s}^2  \le C ( 1 + \|u\|_{H^s}^2 + \|u\|_{H^s}^3 )  
\end{align}
Let $h(t) = \|u(t)\|_{H^s}^{2}$ so the equation reduce to $ \frac{d}{dt}h(t) \le C \Big(1 + h(t) + h(t)^{\frac{3}{2}} \Big)  $ Now there are two cases
\begin{itemize}
    \item if $ h(t) \le 1, $ then we have $ 1 + h(t) + h(t)^{\frac{3}{2}}  \le 3.$

    \item if $ h(t) \ge 1, $ the we have $ 1 \le h(t) \le h(t)^{\frac{3}{2}} \le h(t)^2,$
\end{itemize}
In any case we get the following 
\begin{align}\label{ineq_h}
& \frac{d}{dt}h(t) \le 3C \Big( 1 + h(t)^2 \Big) \notag \\
&\implies   \int_0^t \frac{d h(t)}{ 1 + h(t)^2} \le 3C \int_0^t dt  \implies \tan^{-1}\Big(h(t)\Big) - \tan^{-1}\Big(h(0)\Big) \le 3Ct \notag \\
& \implies h(t) \le \tan\Big( 3Ct + \tan^{-1}\Big(h(0)\Big) \Big)
\end{align}
From  \eqref{ineq_h} we get, $ \|u(t)\|_{H^s}^2 \le \tan\Big( 3Ct + \tan^{-1}\Big( \|u_0\|_{H^s}^2 \Big)   ,$ Now $ \|u_0\|_{H^s}^2 \in \mathbb{R}^{\ge 0}.$ then $$0 \le \tan^{-1} \Big( \|u_0\|_{H^s}^2 \Big) < \frac{\pi}{2} $$
Choose $T_*( u_0, f , \varphi) > 0,$ such that $$0 \le 3C T_* + \tan^{-1} \Big( \|u_0\|_{H^s}^2 \Big) < \frac{\pi}{2},$$
then $\forall t \le T_*(u_0, f, \varphi)$ we have $$3C t + \tan^{-1} \Big( \|u_0\|_{H^s}^2 \Big) \le 3C T_* + \tan^{-1} \Big( \|u_0\|_{H^s}^2 \Big)$$ so the solution exists $\forall t \in [0, T_*].$ Hence, for $s > \frac{3}{2}$ we have $u \in C([0, T_*], H^s(\mathbb{T})).$

\subsection{Proof of Proposition \ref{prp_cntity}} As we have seen the equation is wellposed for $s > \frac{3}{2},$ so through out this prove we will strict our self in the space $H^s(\mathbb{T})$ for $s > \frac{3}{2}.$ we want to prove for given $u_0, v_0 \in H^{s + 1}(\mathbb{T}) $
\begin{align*}
\|\mathcal{R}_t(u_0, 0, g) - \mathcal{R}_t(v_0, 0, g)\|_{H^s}\le c \|u_0 - v_0\|_{H^s} 
\end{align*}
\begin{proof}
As we have defined in Section $2,$ $\mathcal{R}_t(u_0, 0, g)$ and $\mathcal{R}_t(v_0, 0, g)$ are the respective solutions at time $t$ of the equation 
\begin{align}\label{u_equn}
\begin{cases}
u_t = -u\partial_xu-(1-\partial_{xx})^{-1}\Big[\partial_x u + 2u\partial_x u + \partial_x u\partial_{xx} u\Big] + (1 - \partial_{xx})^{-1}g  \\
u(0, x) = u_0(x)
\end{cases}
\end{align}
and 
\begin{align}\label{v_equn}
\begin{cases}
v_t = -v\partial_xv-(1-\partial_{xx})^{-1}\Big[\partial_x v + 2v\partial_x v + \partial_x v\partial_{xx} v\Big] + (1 - \partial_{xx})^{-1}g  \\
v(0, x) = v_0(x)
\end{cases}
\end{align}

Then by the previous Proposition there exists $T^1_* , T^2_* > 0$ such that $u \in C([0, T^1_*]; H^{s + 1}(\mathbb{T}))$ and  $v \in C([0, T^2_*]; H^{s + 1}(\mathbb{T})).$ Now subtracting \eqref{v_equn} from \eqref{u_equn} and setting $w = u - v, f = u + v $ and  $F(u) := (1-\partial_{xx})^{-1}\Big[\partial_x v + 2v\partial_x v + \partial_x v\partial_{xx} v\Big]$ consider the equation 
\begin{align}\label{w_equn}
\begin{cases}
w_t = -\frac{1}{2}\partial_x (fw)- \Big[ F(u) - F(v) \Big] \\
w(0, x) = w_0(x)
\end{cases}
\end{align}
As we have discussed in the proof of Proposition \ref{prp_existance}, we have same problem with term $\partial_x(fw)$ so we again consider the following mollified i.v.p.
\begin{align}\label{mol_w_equn}
\begin{cases}
w_t = -\frac{1}{2}\partial_x [(J_{\varepsilon}f)(J_{\varepsilon}w)]- \Big[ F(u) - F(v) \Big] \\
w(0, x) = w_0(x)
\end{cases}
\end{align}
Calculating the $H^s$ energy of $w$ gives the equation
\begin{align}\label{w_est}
&\frac{d}{dt}  \|J_{\epsilon} w\|_{H^s}^2  =  \frac{d}{dt}\langle \Lambda^s J_{\epsilon}w , \Lambda^s J_{\epsilon}w \rangle_{L^2} = 2 \langle \Lambda^s \partial_t J_{\epsilon}w, \Lambda^s J_{\epsilon}w \rangle_{L^2}\notag \\ 
& \quad = -  \langle \Lambda^s  J_{\epsilon} \partial_x [(J_{\varepsilon} f) (J_{\varepsilon} w)] , \Lambda^s J_{\epsilon}w \rangle_{L^2} -2 \langle \Lambda^s  J_{\epsilon} [F(u) - F(v)] , \Lambda^s J_{\epsilon}w \rangle_{L^2} 
\end{align}
Consider the first term of \eqref{w_est}
 first commuting the $J_{\varepsilon}$ and then using \eqref{self_adj}, arriving at
\begin{align}\label{w_est_1}
&\left|\int_{\mathbb{T}} \Lambda^s \partial_x [(J_{\varepsilon} f) (J_{\varepsilon} w)] \textit{ . } \Lambda^s J^2_{\epsilon}w dx\right| = \Bigg|\int_\mathbb{{T}} [\Lambda^s \partial_x , J_{\varepsilon} f] J_{\varepsilon} w \textit{  }\Lambda^s J_{\varepsilon}^2 w dx \notag \\
& \hspace{4cm}+ \int_{\mathbb{T}} J_{\varepsilon} f \partial_{x} \Lambda^s J_{\varepsilon} w \textit{ } \Lambda^s J_{\varepsilon}^2 w dx\Bigg|
\end{align}
Now consider the first part of \eqref{w_est_1} and apply cauchy-Schwarz inequality
\begin{align}\label{w_est_1_a}
\Bigg|\int_\mathbb{{T}} [\Lambda^s \partial_x , J_{\varepsilon} f] J_{\varepsilon} w \textit{  }\Lambda^s J_{\varepsilon}^2 w dx \Bigg| \le \Bigg\|[\Lambda^s \partial_x , J_{\varepsilon} f] J_{\varepsilon} w\Bigg\|_{L^2}  \Bigg\|\Lambda^s J_{\varepsilon}^2 w\Bigg\|_{L^2}  
\end{align}
after applying \eqref{Kato_Ponce_estimate} on first part \eqref{w_est_1_a} and use $\|J_{\varepsilon}w\|_{H^s} \le \|w\|_{H^s},$ we get
\begin{align}\label{w_est_1_a_fnl}
\Bigg|\int_\mathbb{{T}} [\Lambda^s \partial_x , J_{\varepsilon} f] J_{\varepsilon} w \textit{  }\Lambda^s J_{\varepsilon}^2 w dx \Bigg|  & \le \Bigg( \|\Lambda^s \partial_x J_{\varepsilon}f\|_{L^2} \|J_{\varepsilon}w\|_{L^{\infty}} + \|\partial_x J_{\varepsilon} f\|_{L^{\infty}} \|\Lambda^{s - 1} \partial_x J_{\varepsilon}w\|_{L^2} \Bigg) \|w\|_{H^s} \notag \\
&\le \|f\|_{H^{s + 1}} \|w\|_{H^s}^2
\end{align}
where we have used as $f \in H^{s+1}$ then $\|f_x\|_{L^{\infty}} \le \|f_x\|_{H^s} \le \|f\|_{H^{s+1}} .$

Considering the second term of \eqref{w_est_1} we have and using \eqref{self_adj}
\begin{align}\label{w_est_1_b}
\Bigg|\int_{\mathbb{T}} J_{\varepsilon} f \partial_{x} \Lambda^s J_{\varepsilon} w \textit{ } \Lambda^s J_{\varepsilon}^2 w dx\Bigg|  & = \Bigg|\int_{\mathbb{T}} J_{\varepsilon} f \partial_{x} \Lambda^s J_{\varepsilon} w \textit{ . } J_{\varepsilon} \Lambda^s J_{\varepsilon} w dx\Bigg| = \Bigg|\int_{\mathbb{T}} J_{\varepsilon}^2 f \partial_{x} \Lambda^s J_{\varepsilon} w \textit{ . } \Lambda^s J_{\varepsilon} w dx\Bigg| \notag \\
& \quad = \Bigg|\int_{\mathbb{T}} J_{\varepsilon} f \partial_{x} \Big(\Lambda^s J_{\varepsilon} w \Big)^2 dx \Bigg| \le \|\partial_x f\|_{L^{\infty}} \|w\|_{H^s}^2  \le \|f\|_{H^{s + 1}} \|w\|_{H^s}^2.
\end{align}

Now consider second term of \eqref{w_est},
\begin{align}\label{w_est_2}
\Bigg|\int_{\mathbb{T}} \Lambda^s  J_{\epsilon} [F(u) - F(v)]  \Lambda^s J_{\epsilon}w  dx\Bigg| \le \|F(u) - F(v)\|_{H^s} \|w\|_{H^s} 
\end{align}
Now $F(u) = (1-\partial_{xx})^{-1}\Big[\partial_x v + 2v\partial_x v + \partial_x v\partial_{xx} v\Big] $ So 
\begin{align}\label{w_est_2_a}
\|F(u) - F(v)\|_{H^s} & = \|(1 - \partial_{xx})^{-1} \partial_x (u-v) \|_{H^s} + \|(1 - \partial_{xx})^{-1} \partial_x (u^2 - v^2) \|_{H^s} \notag \\
& \hspace{2cm} + \frac{1}{2} \|(1 - \partial_{xx})^{-1} \partial_x (u_x^2 - v_x^2) \|_{H^s} \notag \\
& \quad \le \| u - v\|_{H^{s - 1}} + \|u^2 - v^2\|_{H^{s - 1}} + \|u_x^2 - v_x^2\|_{H^{s-1}} \notag \\
& \quad \le \| u - v\|_{H^s} + \|(u - v) (u + v)\|_{H^{s - 1}} + \|(u_x - v_x) (u_x + v_x)\|_{H^{s-1}} \notag \\
& \quad  \le \| u - v\|_{H^s} + \|(u - v)\|_{H^{s -1}} \|(u + v)\|_{H^{s - 1}} + \|(u_x - v_x)\|_{H^{s-1}} \|(u_x + v_x)\|_{H^{s-1}} \notag \\
& \quad \le \Big( 1 + \|u + v\|_{H^{s - 1}} + \|u + v\|_{H^s}  \Big) \|u - v\|_{H^s} \notag \\
& \quad \le c_0\Big(1 + \|f\|_{H^{s + 1}}\Big) \|w\|_{H^s}
\end{align}
Now using \eqref{w_est_1} - \eqref{w_est_2_a} from \eqref{w_est} and taking $\varepsilon \to 0,$ we have
\begin{align}\label{w_est_fnl}
 \frac{d}{dt} \|w\|_{H^s}^2 \le c_1 \Big(1 + \|f\|_{H^{s + 1}} \Big)\|w\|_{H^s}^2   
\end{align}
Take $0 < T < \min\left\{\frac{T^1_*}{2}, \frac{T^2_*}{2}\right\},$ then $f \in C([0, T]; H^{s+1}),$ now from \eqref{w_est_fnl} for  any $ t \le T,$ there exists $K := c_1\Big(1 + \|f\|_{H^{s + 1}} \Big)$ such that
\begin{align*}
\|w(t)\|_{H^s} \le e^{KT}\|w(0)\|_{H^s} \text{ for all } t \in [0, T].
\end{align*}
Hence we have obtain \eqref{lips_property}.

\end{proof}

\subsection{Proof of Proposition \ref{prp_limit}}Let $ s > \frac{3}{2},$ for all $ u_0, \varphi \in \eta_0 \in H^{s + 1}(\mathbb{T})$, then there exists $\delta_0 > 0$ such that $(\delta^{-\frac{1}{2}} \varphi, \delta^{-1} \eta_0) \in \widehat{\Theta}(u_0, T)$ for any $\delta \in (0,\delta_0)$, the following limit holds at $t = \delta$
\begin{align*}
\mathcal{R}_{\delta}(u_0, \delta^{-\frac{1}{2}}\varphi, \delta^{-1}\eta_0) \rightarrow u_0 - \varphi\varphi_x + (1 - \partial_{xx})^{-1}\left(\eta_0 - 2\varphi\varphi_x - \varphi_x\varphi_{xx}\right), \text{in}\ H^s(\mathbb{T})  \quad \text{ as }\ \delta \rightarrow 0. 
\end{align*}
\begin{proof} Let us consider the equation
 \begin{align}\label{equn_P1}
\begin{cases}
u_t =-(1-\partial_{xx})^{-1}\Big[(u+\delta^{-\frac{1}{2}}\varphi)_x + 2(u+\delta^{-\frac{1}{2}}\varphi)(u+\delta^{-\frac{1}{2}}\varphi)_x + (u+\delta^{-\frac{1}{2}}\varphi)_x(u+\delta^{-\frac{1}{2}}\varphi)_{xx} \\ \hspace{8cm}-\delta^{-1}\eta_0\Big] -(u+\delta^{-\frac{1}{2}}\varphi)(u+\delta^{-\frac{1}{2}}\varphi)_x \\
u(0,x) = u_0
\end{cases}
\end{align}
Make a time substitution and consider the functions $$v(t) := u(\delta t),$$ $$w(t) := u_0 + t\Big\{-\varphi\varphi_x + (1 - \partial_{xx})^{-1}\left(\eta_0 - 2\varphi\varphi_x - \varphi_x\varphi_{xx}\right)\Big\},$$ and $$z(t) := v(t) - w(t)$$
where $t \in [0, 1].$ Then it is not difficult to see that $z$ is a solution of the following system 
\begin{align}
\begin{cases}\label{equn_P2}
z_t =-\delta\Bigg\{(1-\partial_{xx})^{-1}\Big[(z + w +\delta^{-\frac{1}{2}}\varphi)_x + 2(z + w )(z + w +\delta^{-\frac{1}{2}}\varphi)_x + 2\delta^{-\frac{1}{2}}\varphi(z + w)_x +\\ \hspace{2cm}(z + w )_x(z + w +\delta^{-\frac{1}{2}}\varphi)_{xx} +\delta^{-\frac{1}{2}} \varphi_x(z + w)_{xx}\Big] +(z + w )(z + w +\delta^{-\frac{1}{2}}\varphi)_x + \\ \hspace{8cm}\delta^{-\frac{1}{2}} \varphi(z + w)_{x}\Bigg\} \\
z(0,x) = 0.
\end{cases}
\end{align}
Our aim to show $\|z(t)\|^2_{H^s(\mathbb{T})} \le C\delta. \quad \forall t \in [0, 1].$ In particular $\|z(1)\|_{H^s}^2 \to 0$ as $\delta \to 0.$  We can write the above equation like an ODE as \eqref{inv_eqn_main}
\begin{equation}\label{z_ODE}
\begin{cases}
z_t = -\delta\Big\{(z + w )(z + w +\delta^{-\frac{1}{2}}\varphi)_x +\delta^{-\frac{1}{2}} \varphi(z + w)_{x} + \tilde{F}(u) \Big\} \\
z(0,x)= 0.
\end{cases}
\end{equation}
where
\begin{align}\label{tilde_F(u)}
\tilde{F}(u) & = \Lambda^{-2}\Big[ (z + w +\delta^{-\frac{1}{2}}\varphi)_x + 2(z + w )(z + w +\delta^{-\frac{1}{2}}\varphi)_x + 2\delta^{-\frac{1}{2}}\varphi(z + w)_x \notag \\  & \hspace{5cm} + (z + w )_x(z + w +\delta^{-\frac{1}{2}}\varphi)_{xx} +\delta^{-\frac{1}{2}} \varphi_x(z + w)_{xx} \Big] 
\end{align}
As we have discussed in the proof of Proposition \ref{prp_cntity} to consider \eqref{z_ODE} like an ODE in the Banach space we  the nonlinear term of the equation. i.e using the Friedrichs mollifier $J_{\varepsilon},$  we mollified the term $z z_x $ as $(J_{\varepsilon} z) (J_{\varepsilon} z_x).$ So we do the followings : 

\textbf{The Mollified i.v.p.} Next , we study the following mollified version of problem  of \eqref{z_ODE}
\begin{align}\label{z_mollified_ode}
\begin{cases}
z_t = - \delta\left\{\left[(J_{\varepsilon} z) (J_{\varepsilon} z_x) + w w_x + z w_x + w z_x + \delta^{-\frac{1}{2}} z\varphi_x + \delta^{-\frac{1}{2}} w\varphi_x\right] + \delta^{-\frac{1}{2}} \varphi(z + w)_{x} + \tilde{F}(u)\right\} \\
z(0, x) = 0
\end{cases}
\end{align}
Similarly we have,
\begin{align}\label{z_engy_est}
&\frac{d}{dt}  \|J_{\epsilon} z\|_{H^s}^2  =  \frac{d}{dt}\langle \Lambda^s J_{\epsilon}z , \Lambda^s J_{\epsilon}z \rangle_{L^2} = 2 \langle \Lambda^s \partial_t J_{\epsilon}z , \Lambda^s J_{\epsilon}z \rangle_{L^2}\notag \\ 
& \quad = - 2 \delta \Big\langle \Lambda^s  J_{\epsilon} [(J_{\varepsilon} z) (J_{\varepsilon} z_x)] , \Lambda^s J_{\epsilon}z \Big\rangle_{L^2} -2\delta \Big\langle \Lambda^s  J_{\epsilon} [ww_x + z w_x + w z_x + \delta^{-\frac{1}{2}} z\varphi_x + \delta^{-\frac{1}{2}} w\varphi_x ] , \Lambda^s J_{\epsilon}z \Big\rangle_{L^2} \notag \\
& \hspace{2cm}\quad - 2\delta^{\frac{1}{2}} \Big\langle \Lambda^s  J_{\epsilon}  \varphi(z + w)_{x}   , \Lambda^s J_{\epsilon}z \Big\rangle_{L^2} - 2 \delta \Big\langle \Lambda^s  J_{\epsilon} \tilde{F}(u) , \Lambda^s J_{\epsilon}z \Big\rangle_{L^2} 
\end{align}
Now Consider the first term of \eqref{z_engy_est}, and doing the same as \eqref{est_1_a_fnl} and \eqref{est_1_b_fnl} we have
\begin{align}\label{z_est_1}
&\left|\int_{\mathbb{T}} \Lambda^s [(J_{\varepsilon} z) (J_{\varepsilon} z_x)] \textit{ . } \Lambda^s J^2_{\epsilon}z dx\right| \le \|z_x\|_{L^{\infty}} \|z\|_{H^s}^2
\end{align}
For the second term of \eqref{z_engy_est}, we have
\begin{align}\label{z_est_2}
&\left|\int_{\mathbb{T}} \Lambda^s J_{\varepsilon} [ww_x + z w_x + w z_x + \delta^{-\frac{1}{2}} z\varphi_x + \delta^{-\frac{1}{2}} w\varphi_x ] \textit{ . } \Lambda^s J_{\epsilon}z dx\right| \notag \\
& \le \left|\int_{\mathbb{T}} \Lambda^s J_{\varepsilon} (w w_x )\textit{ . } \Lambda^s J_{\epsilon}z dx\right| +  \left|\int_{\mathbb{T}} \Lambda^s J_{\varepsilon} (z w_x )\textit{ . } \Lambda^s J_{\epsilon}z dx\right| + \left|\int_{\mathbb{T}} \Lambda^s J_{\varepsilon} (w z_x ) \textit{ . } \Lambda^s J_{\epsilon}z dx\right|\notag \\
&\hspace{3cm}+ \left|\int_{\mathbb{T}} \Lambda^s J_{\varepsilon}  \delta^{-\frac{1}{2}} (z\varphi_x)  \textit{ . } \Lambda^s J_{\epsilon}z dx\right| + \left|\int_{\mathbb{T}} \Lambda^s J_{\varepsilon}  \delta^{-\frac{1}{2}} (w\varphi_x ) \textit{ . } \Lambda^s J_{\epsilon}z dx\right| \notag \\
\end{align}
For the first term of \eqref{z_est_2} applying Cauchy Schwarz inequality and  for $s > \frac{3}{2}, H^s$ is  an algebra, we have
\begin{align}\label{z_est_2_a}
\left|\int_{\mathbb{T}} \Lambda^s J_{\varepsilon} (w w_x )\textit{ . } \Lambda^s J_{\epsilon}z dx\right| \le \Big\|\Lambda^s J_{\varepsilon} (w w_x )\Big\|_{L^2} \Big\|\Lambda^s J_{\epsilon}z\Big\|_{L^2} \le \Big\|(w w_x )\Big\|_{H^s} \Big\|z\Big\|_{H^s} \le \Big\| w\Big\|_{H^s} \Big\| w_x \Big\|_{H^s} \Big\|z\Big\|_{H^s} 
\end{align}
Similarly for the second term of \eqref{z_est_2} applying Cauchy Schwarz inequality and  for $s > \frac{3}{2}, H^s$ is  an algebra, we have
\begin{align}\label{z_est_2_b}
\left|\int_{\mathbb{T}} \Lambda^s J_{\varepsilon} (z w_x )\textit{ . } \Lambda^s J_{\epsilon}z dx\right| \le \Big\|\Lambda^s J_{\varepsilon} (z w_x )\Big\|_{L^2} \Big\|\Lambda^s J_{\epsilon}z\Big\|_{L^2} \le \Big\|(z w_x )\Big\|_{H^s} \Big\|z\Big\|_{H^s} \le \Big\| w_x \Big\|_{H^s} \Big\|z\Big\|_{H^s}^2  
\end{align}

Now consider the third term of \eqref{z_est_2}, we have 
\begin{align}
&\left|\int_{\mathbb{T}} \Lambda^s J_{\varepsilon} (w z_x ) \textit{ . } \Lambda^s J_{\epsilon}z dx\right| =  \left|\int_{\mathbb{T}} \Lambda^s (w z_x ) \textit{ . } \Lambda^s J_{\epsilon}^2z dx\right|  = \left|\int_{\mathbb{T}} \Big\{\Lambda^s (w z_x ) - w(\Lambda^s z_x) + w(\Lambda^s z_x) \Big\} \textit{ . } \Lambda^s J_{\epsilon}^2z dx\right| \notag \\
& \quad \le \left|\int_{\mathbb{T}} [\Lambda^s , w] z_x \textit{ . } \Lambda^s J_{\epsilon}^2z dx\right| + \left|\int_{\mathbb{T}} w(\Lambda^s z_x ) \textit{ . } \Lambda^s J_{\epsilon}^2z dx\right| \notag \\
& \quad \le \|[\Lambda^s , w] z_x\|_{L^2} \|\Lambda^s J_{\epsilon}^2z\|_{L^2} + \left|\int_{\mathbb{T}} J_{\varepsilon}(w \partial_x \Lambda^s z ) \textit{ . } \Lambda^s J_{\epsilon}z dx\right| \notag \\
&\quad \le \|[\Lambda^s , w] z_x\|_{L^2}\|J_{\varepsilon}z\|_{H^s} +  \left|\frac{1}{2} \int_{\mathbb{T}} \Big([J_{\varepsilon}, w] \Lambda^s z_{x} \Big) \Lambda^s J_{\varepsilon} z dx\right| + \left|\frac{1}{2} \int_{\mathbb{T}} \partial_x w \Big( J_{\varepsilon} \Lambda^s z \Big) \Lambda^s J_{\varepsilon} z dx \right|  \textit{ ( same as \eqref{est_1_b}) } \notag  \\
& \quad \le \Big( \|\Lambda^s w\|_{L^2} \|z_x\|_{L^{\infty}} + \|w_x\|_{L^{\infty}} \|\Lambda^{s - 1} z_x\|_{L^2} \Big) \|z\|_{H^s} + \Big(\|w_x\|_{L^{\infty}} \|z\|_{H^s} \Big)\|z\|_{H^s} + \|w_x\|_{L^{\infty}} \|z\|_{H^s}^2 \notag \\
& \quad \le \Big( \| w\|_{H^s} \|z_x\|_{L^{\infty}} + \|w_x\|_{L^{\infty}} \| z\|_{H^s} \Big) \|z\|_{H^s}  +  \|w_x\|_{L^{\infty}} \|z\|_{H^s}^2  \label{z_est_2_c}
\end{align}
As we have done for the second term of \eqref{z_est_2}, i.e like \eqref{z_est_2_b} we have for the fourth and the fifth terms
\begin{align}\label{z_est_2_d}
\left|\int_{\mathbb{T}} \Lambda^s J_{\varepsilon}  \delta^{-\frac{1}{2}} (z\varphi_x)  \textit{ . } \Lambda^s J_{\epsilon}z dx\right| \le \delta^{-\frac{1}{2}} \Big\|\varphi_x\Big\|_{H^s} \Big\|z\Big\|_{H^s}^2    
\end{align}
and
\begin{align}\label{z_est_2_e}
\left|\int_{\mathbb{T}} \Lambda^s J_{\varepsilon}  \delta^{-\frac{1}{2}} (w\varphi_x)  \textit{ . } \Lambda^s J_{\epsilon}z dx\right| \le \delta^{-\frac{1}{2}} \Big\|w\Big\|_{H^s}\Big\|\varphi_x\Big\|_{H^s} \Big\|z\Big\|_{H^s}    
\end{align}
Now consider the third term of \eqref{z_engy_est}, we have
\begin{align}
&\Big| \int_{\mathbb{T}} \Lambda^s  J_{\epsilon}  \varphi(z + w)_{x} \textit{ . } \Lambda^s J_{\epsilon}z dx  \Big| \le   \Big| \int_{\mathbb{T}} \Lambda^s  J_{\epsilon}  \varphi z_{x} \textit{ . } \Lambda^s J_{\epsilon}z dx  \Big| + \Big| \int_{\mathbb{T}} \Lambda^s  J_{\epsilon}  \varphi w_{x} \textit{ . } \Lambda^s J_{\epsilon}z dx  \Big|\notag \\
& \le \Big(\|\varphi_x\|_{L^{\infty}} \|z\|_{H^s} \Big)\|z\|_{H^s} + \|\varphi_x\|_{L^{\infty}} \|z\|_{H^s}^2 + \|\varphi\|_{H^s} \|w_x\|_{H^s} \|z\|_{H^s} \textit{ (using \eqref{est_1_b}) } \notag \\
& \le \|\varphi_x\|_{L^{\infty}} \|z\|_{H^s}^2 + \|\varphi\|_{H^s} \|w_x\|_{H^s} \|z\|_{H^s} \label{z_est_3}
\end{align}
Now we bound the  remaining term \eqref{z_engy_est}
\begin{align}\label{z_est_4}
\Big| \int_{\mathbb{T}} \Lambda^s  J_{\epsilon} \Tilde{F}(u) \textit{ . } \Lambda^s J_{\epsilon}z dx  \Big| \le   \|\tilde{F}(u) \|_{H^s} \|z\|_{H^s} 
\end{align} 
Now $\tilde{F}(u)$ is given by \eqref{tilde_F(u)} and we know $\Lambda^{-2} \partial_x : H^{s - 1} \to H^s$ is a bounded operator and assuming $s > \frac{3}{2}$ then $H^{s - 1}$ forms an algebra . i.e $$\|u v \|_{H^{s - 1}} \le \|u\|_{H^{s-1}} \|v\|_{H^{s-1}}.$$  So 
\begin{align}\label{nrm_tilde_F(u)}
\Big\| & \Lambda^{-2}\Big[ (z + w +\delta^{-\frac{1}{2}}\varphi)_x + 2(z + w )(z + w +\delta^{-\frac{1}{2}}\varphi)_x + 2\delta^{-\frac{1}{2}}\varphi(z + w)_x \notag \\
& \hspace{5cm}+ (z + w )_x(z + w +\delta^{-\frac{1}{2}}\varphi)_{xx} +\delta^{-\frac{1}{2}} \varphi_x(z + w)_{xx} \Big]  \Big\|_{H^s}  \notag \\
& \le \|\Lambda^{-2} \partial_x z \|_{H^s} + \|\Lambda^{-2} \partial_x w \|_{H^s} + \delta^{-\frac{1}{2}} \|\Lambda^{-2} \partial_x \varphi \|_{H^s} + \|\Lambda^{-2} \partial_x (z^2) \|_{H^s} + \|\Lambda^{-2} \partial_x (w^2) \|_{H^s} \notag \\
& \quad + \|\Lambda^{-2} \partial_x (zw) \|_{H^s} + \delta^{-\frac{1}{2}} \|\Lambda^{-2} \partial_x (z \varphi) \|_{H^s} + \delta^{-\frac{1}{2}} \|\Lambda^{-2} \partial_x (w \varphi) \|_{H^s} + \|\Lambda^{-2} \partial_x (z_x ^2) \|_{H^s} + \|\Lambda^{-2} \partial_x (w_x ^2) \|_{H^s} \notag \\
& \quad \hspace{3cm}+ \|\Lambda^{-2} \partial_x (z_x w_x) \|_{H^s} + \delta^{-\frac{1}{2}}\|\Lambda^{-2} \partial_x (z_x \varphi_x) \|_{H^s} + \delta^{-\frac{1}{2}}\|\Lambda^{-2} \partial_x (w_x \varphi_x) \|_{H^s} \notag \\
& \le \|z\|_{H^{s - 1}} + \|w\|_{H^{s - 1}} +\delta^{-\frac{1}{2}} \|\varphi\|_{H^{s - 1}} + \|z^2\|_{H^{s - 1}} + \|w^2\|_{H^{s - 1}} + \|(zw)\|_{H^{s - 1}} +\delta^{-\frac{1}{2}} \|(z \varphi)\|_{H^{s - 1}} \notag \\
& \quad + \delta^{-\frac{1}{2}} \|(w \varphi)\|_{H^{s - 1}} + \|z_x ^2\|_{H^{s - 1}} + \|w_x ^2\|_{H^{s - 1}} + \|(z_x w_x)\|_{H^{s - 1}} + \delta^{-\frac{1}{2}} \|(z_x \varphi_x)\|_{H^{s - 1}} + \delta^{-\frac{1}{2}} \|(w_x \varphi_x)\|_{H^{s - 1}} \notag \\
& \le \|z\|_{H^s} + \|w\|_{H^s} + \delta^{-\frac{1}{2}} \|\varphi\|_{H^s} + \|z\|_{H^{s - 1}} \|z\|_{H^{s - 1}} + \|w\|_{H^{s - 1}} \|w\|_{H^{s - 1}} + \|z\|_{H^{s - 1}} \|w\|_{H^{s - 1}}\notag \\
& \quad  + \delta^{-\frac{1}{2}} \|z\|_{H^{s - 1}} \|\varphi\|_{H^{s - 1}}+ \delta^{-\frac{1}{2}} \|w\|_{H^{s - 1}} \|\varphi\|_{H^{s - 1}} + \|z_x\|_{H^{s - 1}} \|z_x\|_{H^{s - 1}} + \|w_x\|_{H^{s - 1}} \|w_x\|_{H^{s - 1}}\notag \\
& \quad \quad  + \|z_x\|_{H^{s - 1}} \|w_x\|_{H^{s - 1}} + \delta^{-\frac{1}{2}}\|z_x\|_{H^{s - 1}} \|\varphi_x\|_{H^{s - 1}}+ \delta^{-\frac{1}{2}}\|w_x\|_{H^{s - 1}} \|\varphi_x\|_{H^{s - 1}} \notag \\
& \le \|z\|_{H^s} + \|w\|_{H^s} + \delta^{-\frac{1}{2}} \|\varphi\|_{H^s} + \|z\|_{H^s}^2 + \|w\|_{H^s}^2 + \|z\|_{H^s} \|w\|_{H^s} + \delta^{-\frac{1}{2}} \|z\|_{H^s} \|\varphi\|_{H^s}+ \delta^{-\frac{1}{2}} \|w\|_{H^s} \|\varphi\|_{H^s} \notag \\
& \quad \quad + \|z\|_{H^s}^2 + \|w\|_{H^s}^2 + \|z\|_{H^s} \|w\|_{H^s} + \delta^{-\frac{1}{2}}\|z\|_{H^s} \|\varphi\|_{H^s}+ \delta^{-\frac{1}{2}}\|w\|_{H^s} \|\varphi_x\|_{H^s}
\end{align}
So from \eqref{z_engy_est} using Inequalities \eqref{z_est_1} - \eqref{nrm_tilde_F(u)} applying Young's Inequality and assuming $\varphi ,\  w \in H^{s + 1}$ and for $\delta$ very small then $ \delta < \delta^{\frac{1}{2}} .$ Now taking $\varepsilon \to 0$  we get
\begin{align}\label{z_engy_est_fnl}
& \frac{d}{dt}  \| z\|_{H^s}^2  \le C \delta^{\frac{1}{2}} ( 1 + \|z\|_{H^s}^2 + \|z\|_{H^s}^3 )  
\end{align}
Since $z(0) = 0$ 
then similarly as \eqref{ineq_h} we have $$\|z(t)\|_{H^s}^2 \le \tan\Big( 3C \delta^{\frac{1}{2}} t \Big)$$  
So $0 \le 3C \delta^{\frac{1}{2}} t < \frac{\pi}{2} $ \text{ i.e } $0 \le  t < \frac{\pi}{2} \Big(\frac{1}{3C \delta^{\frac{1}{2}}}\Big) ,$ Choose $\delta_0 \in (0, 1)$ such that $\frac{\pi}{2} \Big(\frac{1}{3C \delta_0^{\frac{1}{2}}}\Big) > 1.$

So $T_* > 1,$  then the solution exists $(0, T_*), \forall t \in (0, \delta_0).$
Now
\begin{align}\label{ineq_z}
&\|z(t)\|_{H^s}^2 \le \tan\Big( 3C \delta^{\frac{1}{2}} t \Big) \notag \\
&\|z(1)\|_{H^s}^2 \le \tan\Big( 3C \delta^{\frac{1}{2}} \Big)  \text{   Since } t \le 1. \notag \\
& \implies \|z(1)\|_{H^s}^2 \to 0, \text{  as  } \delta \to 0.
\end{align}
In particular, $u(\delta) = v(1) \to w(1)$ as $\delta \to 0,$ and then we obtain the required limit.

\end{proof}

\section{Appendix}
Keeping in our mind that using $\mathcal{H}_0$ valued control we can reach very close to  any element of $u_0 + \mathcal{H}_0$, In this section  we will see details construction of control $\widehat\eta \in L^2(0, t; \mathcal{H}_0)$ depending on target  so that using this control we can reach very close to any element of $u_0 + \mathcal{H}_1$ starting form $u_0 \in H^{s + 1}(\mathbb{T})$ in time $t > 0.$

More precisely, take an element $u_0 + \Tilde{\eta} \in u_0 + \mathcal{H}_1$ where $$\Tilde{\eta} = \eta - \sum _{i=1}^{2} \varphi_{i}\partial_x\varphi_i - (1 - \partial_{xx})^{-1} \sum _{i=1}^{2} \big(2\varphi_i\partial_x\varphi_i + \partial_x \varphi_i \partial_{xx}\varphi_i \big) $$ where $\varphi_1 , \varphi_2 \in \mathcal{H}_0$ for which we will construct the control $\hat{\eta}.$ we will denote $\mathcal{N}(\varphi) := \varphi\partial_x\varphi - (1 - \partial_{xx})^{-1} \big(2\varphi\partial_x\varphi + \partial_x \varphi \partial_{xx}\varphi \big).$

\begin{center}
\includegraphics[scale= 0.3]{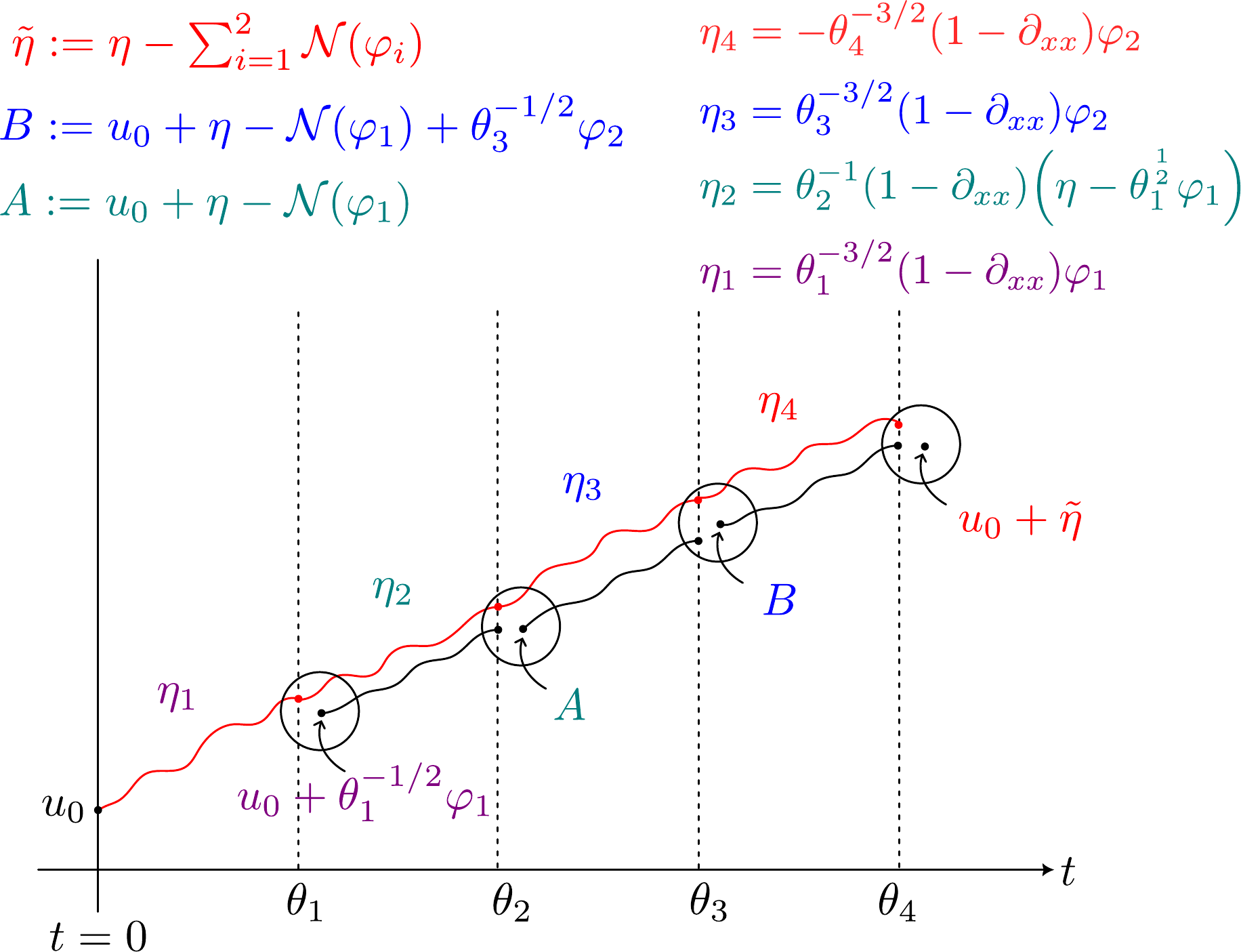}
\end{center}

See the above figure, starting from $u_0 $ using the control $\eta_1$ for the time interval $[0, \theta_1] $ we will reach close to $u_0 + \theta_1^{-\frac{1}{2}} \varphi_1$ at time $t = \theta_1 $ by the Proposition \ref{prp_limit}. Then using the Proposition \ref{prp_cntity} and Proposition \ref{prp_limit} starting from $u_0 + \theta_1^{-\frac{1}{2}} \varphi_1$ by the control $\eta_2 $ we will reach close to the point $ A .$ By similar argument by the control $\widehat\eta : s \to \mathbbm{1}_{[0, \theta_1]} \eta_1 + \mathbbm{1}_{[ \theta_1, \theta_1 + \theta_2]} \eta_2 +  \mathbbm{1}_{[\theta_1 + \theta_2 ,  \theta_1 + \theta_2 + \theta_3]} \eta_3 + \mathbbm{1}_{[\theta_1 +\theta_2 + \theta_3 ,  \theta_1 + \theta_2 + \theta_3 + \theta_4]} \eta_4  $ and using the Lemma \ref{lem_rt} we can reach close to the target $u_0 + \Tilde{\eta}$ starting from $u_0,$ and this control is $\mathcal{H}_0$ valued.

\vspace{5mm}
 \textbf{Acknowledgements.} We would like to express our gratitude
to Mrinmay Biswas for valuable discussions and Subrata Majumder for bringing our attention to this model. The work of the first-named author is supported by MATRICS Research Grant (MTR/2021/00561) . The work of the second-named author is supported by  DST/INSPIRE/$04/2015/002388,$ Third-named author wants
to thank Jiten Kumbhakar and Sandip Samanta for technical support with the figure in the Appendix,
above all grateful to IISER Kolkata for Integrated PhD fellowship.

\bibliographystyle{plain}
\bibliography{reference}

\vspace{8mm}
Department of Mathematics and Statistics, Indian Institute of Science Education and Research Kolkata, Mohanpur – 741 246

Email address:  Shirshendu Chowdhury (shirshendu@iiserkol.ac.in),

 \hspace{2.25cm} Rajib Dutta (rajib.dutta@iiserkol.ac.in), 
 
 \hspace{2.25cm} Debanjit Mondal (wrmarit@gmail.com).

\end{document}